\documentclass[12pt]{article}
 \usepackage{amsmath, amsfonts, amsthm, amssymb, color, hyperref, extarrows, enumitem, nicefrac}

\newtheorem{theorem}{Theorem}[section]

\newtheorem{cor}[theorem]{Corollary}
\newtheorem{lem}[theorem]{Lemma}
\newtheorem{pro}[theorem]{Proposition}

\DeclareMathOperator{\Real}{Re}

\DeclareMathOperator{\dist}{dist}

\numberwithin{equation}{section}

\newtheorem{remark}[theorem]{Remark}

\usepackage{cite}

\hypersetup{
	colorlinks   = true,
	citecolor    = magenta}
	
\usepackage[titletoc]{appendix}

\begin{document}
\title{\vspace{-1.2cm} \bf Uniform estimates for the canonical solution to the $\bar\partial$-equation on product domains  \rm}
\author{Robert Xin Dong  \, \, \quad Yifei Pan   \quad \, \,Yuan Zhang}
\date{}

\maketitle

\begin{abstract}
We obtain uniform estimates for the canonical solution to $\bar\partial u=f$ on the Cartesian product of  bounded planar domains with $C^2$ boundaries, when $f$ is continuous up to the boundary.
This generalizes Landucci's result for the bidisc toward higher dimensional product domains. In particular, it answers an open question of  Kerzman for continuous datum.
\end{abstract}

\renewcommand{\thefootnote}{\fnsymbol{footnote}}
\footnotetext{\hspace*{-7mm} 
\begin{tabular}{@{}r@{}p{16.5cm}@{}}
& Keywords. $\bar{\partial}$-equation, product domain, canonical solution, Green's function\\
& Mathematics Subject Classification. Primary 32A25; Secondary 32A26, 32W05\\

\end{tabular}}

\section{Introduction}

The purpose of this paper is to study uniform estimates for the canonical solution to the $\bar \partial$-equation on the Cartesian product of  bounded planar domains with $C^2$ boundaries. When the data are continuous, we prove the following theorem.

\begin{theorem}\label{main}
Let $\Omega:=D_1\times \cdots  \times D_n \subset \mathbb C^n$, $n\geq 2$, where each $D_j$ is a bounded planar domain with $C^2$ boundary. Then there exists a positive constant $C$ depending only on $\Omega$ such that for any $\bar\partial$-closed $(0, 1)$ form $ f$ continuous up to $\bar\Omega$,  the canonical solution  to  $\bar\partial u = f$ (in the sense of distributions) is continuous on $\Omega$ and satisfies 
$$
\|  u   \|_{ L^\infty(\Omega) }\le C\|  f\|_{ L^\infty(\Omega) }.
$$
\end{theorem}
  
The canonical solution is the unique one that is perpendicular to the kernel of $\bar \partial$, and is sometimes referred to as the $L^2$-minimal solution because it has minimal $L^2$-norm among others. 
The canonical solution exists on any bounded pseudoconvex domain in $\mathbb C^n$ whenever the datum is in $L^2$; 
however, uniform estimates for the solutions
do not exist in general.
Indeed,  Sibony \cite{Sib80}  constructed an example of a $\bar \partial$-closed $(0, 1)$-form $f$ continuous on the closure
of a weakly pseudoconvex domain in   $\mathbb{C}^3$, which has no bounded solution to $\bar \partial u = f$. See also Berndtsson \cite{Be93} for a  counterexample in $\mathbb C^2$. 
More strikingly, Forn\ae ss and Sibony \cite{FS93} constructed a smoothly bounded pseudoconvex domain in $\mathbb{C}^2$ whose boundary is strictly pseudoconvex except at one point, but uniform estimates for the solutions fail to hold.

\medskip

On the other hand, on smoothly bounded strictly pseudoconvex domains  in $\mathbb{C}^n$, Henkin \cite{Hen70} and Grauert and Lieb \cite{G-L70} constructed integral solutions that satisfy the uniform estimates $\|u\|_{{\infty}} \leq C\|f\|_{\infty}$ for $C^{\infty}$ forms $f$.
Kerzman  \cite{Ker71} and later Henkin and Romanov  \cite{HR} extended their results and obtained the $\nicefrac{1} {2}$-H\"{o}lder estimates. See also the works of Lieb and Range \cite{LR}, Range \cite{R78, R86, R90, R20}, Range and Siu \cite{R-S72, R-S73}, as well as \cite{Ker73, K76, LTL93, Be97, Be01, CS11, SV14, GSS19, DLT}, and the references therein.
\medskip

The unit polydisc $\mathbb D^n:=\mathbb D(0,1)^n$ is pseudoconvex with non-smooth boundary. Kerzman in \cite[p. 311-312]{Ker71} proposed an open question that remains unsolved after half a century -- 
{  \it do uniform estimates for $\bar\partial$ hold on $\mathbb{D}^n, n \geq 2$, when the datum  is only assumed to be bounded} ?
Henkin in \cite{Hen71} (see also \cite{F-L-Zh11}) showed that there exist integral solutions with uniform estimates when the  data are $ C^1$ up to $\overline{\mathbb D^2}$; Landucci in \cite{L75} and Bertrams in \cite{B} extended Henkin's result  and obtained  bounded canonical solutions on polydiscs.
  Making use of conformal mappings, these uniform estimates with $C^1$ data are passed onto Cartesian products of  simply connected planar domains with sufficiently smooth boundaries. 
For general product domains, 
the uniform estimates were    
 proved by Fassina and the second author \cite{FP19} when the data are smooth up to the boundaries.    The additional data regularity assumption there is essentially due to the presence of  higher order derivatives of the data in the solution representation in the case of $n\ge 2$.  The solution  in \cite{FP19}  is not canonical.

\medskip
   
Theorem \ref{main} lowers the data regularity assumption to continuity by studying weak solutions, and  generalizes Landucci's canonical solution result \cite{L75} from the bidisc toward higher dimensional product of general bounded planar domains with $C^2$ boundaries. Because of the absence of the explicit Bergman kernel formula for general planar domains, our canonical solution  is not obtained by taking the Bergman projection of a certain solution, but rather is represented in terms  of Green's function along each slice of the product domain.    For differentiable data, the solution operator contains the derivatives of the data  up to order $n-1$ and takes similar forms as in the inspiring works \cite{Hen71, L75,  FP19}.  To achieve uniform estimates, one needs to eliminate  those derivatives making use of the $\bar\partial$-closeness of the data. The crucial ingredients in our proof are the results of Barletta and Landucci \cite{BL91} and Kerzman \cite{Ker76} regarding planar domains, in addition to a refined decomposition of the canonical solution motivated by \cite{L75, FP19}. For 
continuous data, we weaken the regularity assumption by introducing a new canonical integral formula which involves no derivatives of the data. Such type of formula was initially proposed in \cite{FP19} as a non-canonical solution candidate, and we secure such a proposal in Section 5. By a stability result on the canonical solution kernel,  we show for continuous data that our  formula gives a bounded solution that is  also canonical  by carefully exhausting the domain from inside.

\medskip

The estimate in Theorem \ref{main} is optimal in the sense that  the supnorm on the left hand side cannot be replaced by any H\"older norm. In \cite[p. 310-311]{Ker71}, Kerzman and Stein constructed an example of a $\bar\partial$-closed $(0, 1)$ form $f$ that is $L^\infty$ on the bidisc but the equation $\bar\partial u =  f$ admits no H\"older solution.

\medskip

The organization of this paper is as follows.  In Section 2, we summarize a proof of the  uniform estimates for the canonical solutions on planar domains. In  Section 3, we investigate upper bounds for the derivatives of one dimensional canonical solution kernels.
In Section 4, we prove uniform estimates for the canonical solutions on product domains of dimensions two and higher, when the data are differentiable. In Section 5,  the regularity assumption is weakened to continuity, which proves the main theorem completely, and yields Corollary \ref{cor} on uniform estimates for the Bergman projection. Lastly in Appendices \ref{AppA} and \ref {AppB}, we give proofs to a key formula of Barletta-Landucci \cite{BL91}, and upper bounds for Green's function together with its derivatives based on \cite{Ker76}.

 \medskip

After the initial version of this paper was circulated on ArXiv, we observed that our canonical integral formula can be further simplified due to the vanishing property of the solution kernel on the boundary (see Remark \ref{re}). This observation, combined with our main estimate \eqref{G11}, leads to the $L^p$ estimate of the canonical solution operator for $L^p$ data, $1\le p\le \infty$, with the $L^p$ bound independent of $p$. Yuan \cite{Yuan} developed this direction further. More recently, Li \cite{Li} established a different approach showing that Kerzman's question holds true for product domains with $C^{1, \alpha}$ slices, $0<\alpha<1$.

{
\paragraph{Acknowledgements} 

\setlength{\parskip}{-1.5ex}
  
\small
 
We are  grateful to Emil Straube for sharing the unpublished note of Kerzman.}

\section{One dimensional canonical solution to the $\bar{\partial}$-equation}
For a bounded domain $D$ in $ \mathbb C$ with $C^{2}$  boundary,  let $SH(D)$ be the set of subharmonic functions on $D$. The (negative) Green's function $g(z, w)$ with a pole at $w \in D$ is defined as
$$
g(z, w) = \sup \, \{ u(z):  u < 0,\, u \in SH(D),\, \limsup_{\zeta \to w} \, (u(\zeta) - \log |\zeta - w|) < \infty \}.
$$
Denote by $ G(z, w):\equiv - (2\pi)^{-1}  {g(z, w)}$  the positive Green's function. Let
$$
H(w, z):= \frac{1}{2\pi i (z-w)}
$$
stand for the universal Cauchy kernel on $D$. 

For any fixed $w \in D$ consider the Dirichlet problem: 
\begin{equation}\label{Dirichlet}
\begin{cases}
 \Delta L(z) = 0, & z \in D; \\
L(z)=H (w, z), & z \in bD.
\end{cases}
\end{equation}
Then by the Poisson formula one solves \eqref{Dirichlet} uniquely with a solution expressed as 
 \begin{equation} \label{L def}
L(w, z) \equiv L_w(z):= \frac{1}{\pi}\int_{ b D} \frac{1}{\zeta-w} \frac{\partial G(z, \zeta)}{\partial \zeta}d\zeta \left(= - \frac{1}{2\pi i}\int_{ b D} \frac{1}{\zeta-w} \frac{\partial G(z, \zeta)}{\partial \vec{n}_\zeta }ds_\zeta\right).
\end{equation}
According to standard elliptic PDEs theory, $L$  belongs to $ C^\infty( D)\cap  C^{1, \alpha}(\bar D)$ as a function of  $ z $ for any  $\alpha \in (0, 1)$. Here $\vec{n}$ is the unit outer normal vector and $ds$ is the arc length on $b D$.
Regarded as a function on $D\times D$, $L(w, z)$ is holomorphic with respect to all $w$ in $D$.

\vspace{0.15cm}

In \cite{BL91}, Barletta and Landucci rewrote the function $L$ by using Green's second identity as 
\begin{equation} \label{L=H+I}
L(w,z) = H(w, z)  - \frac{1}{\epsilon\pi i}\int_0^{2\pi}G(w+\epsilon e^{it}, z)e^{-it}dt, \quad z\neq w,
\end{equation}
where $\epsilon$ can be chosen as any small number such that 
\begin{equation}\label{epsilon}
    0< \epsilon \le\epsilon (w, z):=2^{-1}{\min} \left \{ {|z-w|} , {\delta(w)}  \right \}.
\end{equation}
Because (\ref{L=H+I}) plays a key role throughout the rest of the paper, we provide its proof for completeness in Appendix \ref{AppA}, which also corrects a minor error in  \cite{BL91}.
 \vspace{0.15cm}
 
For any (0, 1)-form $ f(z) d\bar z$ such that $f \in   L^1 ( D)$, define the operator $\mathbf T$ as
\begin{equation}\label{c1}
\mathbf Tf(w):=\int_D S(w, z) { f} (z)  d\bar z\wedge dz,
\end{equation}
where for any fixed $w\in D$, 
\begin{equation}\label{sl}
    S(w, z): = L(w, z) -H(w, z)\left( = -\frac{1}{\epsilon\pi i}\int_0^{2\pi}G(w+\epsilon e^{it}, z)e^{-it}dt\right). 
\end{equation} 
In particular, 
\begin{equation}\label{S0}
    S(w, z) =0, \ \ \ z\in bD. \ \ 
\end{equation}
The classical Cauchy integral theory
and the holomorphy of $L$ with respect to $w$ imply that 
$$
\frac{\partial  }{\partial  \bar w } (\mathbf T f(w)) =\frac{\partial  }{\partial  \bar w } \left (\frac{1}{2\pi i}\int_{ D}\frac{ { f} (z)}{z-w} dz \wedge d\bar z \right) = { f}(w)
$$
weakly, so $\mathbf T$ is also a solution operator. 
 
 \vspace{0.15cm}
 
In his unpublished note \cite{Ker76}, Kerzman obtained estimates for Green's function on general $C^2$-smooth bounded domains in $\mathbb R^n, n\geq 2$. 

\begin{lem} [Kerzman \cite{Ker76}] \label{Kerzman-Green}

Let $D$ be a  bounded domain with $C^2$ boundary in $\mathbb C$, whose diameter is $d$. Let $\delta(\cdot)$ denote the distance to the boundary $bD$.
Then there exists a positive constant $C$ depending only on $D$ such that for $(w, z) \in \bar D \times \bar D \setminus \{z=w\}$,\\
a) \begin{equation}\label{G}
  G(w, z)\le   \frac{C \delta(w)}{|z-w|}\log\frac{2d}{|z- w|},  \quad G(z, w)\le \frac{C \delta(z)}{|z-w|}\log \frac{2d}{|z-w|};
\end{equation}
b)\begin{equation}\label{g2}
      G(w, z) \le   \frac{C \delta(z)\delta(w)}{|z-w|^2}\log \frac{2d}{|z-w|}.
 \end{equation}
\end{lem}

\vspace{0.15cm}

The proofs of Kerzman's estimates for real dimensions $n\geq 3$ were presented by Krantz in \cite[Propositions 8.2.2, 8.2.6]{K01}. For the sake of completeness, in Appendix \ref{AppB} we shall provide a proof to  Lemma \ref{Kerzman-Green}  on planar domains,  by imitating \cite{K01}. 
In \cite{BL91} Barletta and Landucci   estimated the canonical solution kernel $S$ by using Kerzman's estimates on Green's function. Their idea gives rise to the following result on  planar domains with $C^2$ boundaries.

\begin{pro}\label{|L|}

Let $D$ be a  bounded domain with $C^2$ boundary in $\mathbb C$, whose diameter is d. Then there exists a positive constant $C$ depending only on $D$ such that
\begin{equation}\label{S}
|S(w, z)|\le \frac{C }{|z-w|} \log \frac{2d}{|z-w|} ;   \quad \quad
|S(w, z)|  \le    \frac{C \delta(z)}{|z-w|^2}\log \frac{2d}{|z-w|},
\end{equation}
 for $(w, z) \in   D \times \bar D \setminus \{z=w\}$.
\end{pro}

\begin{proof}
It suffices to estimate the last term in \eqref{L=H+I}. We  first note that  from the definition of $\epsilon$, 
\begin{equation}\label{comp}
| w+\epsilon  e^{it}- z  |\ge | w-z| -\epsilon \ge  \frac{|z-w|}{2}, \ \ \ \text{and}\ \ \ \delta( w+\epsilon  e^{it}  )\le \delta( w  )+\epsilon\le \frac{3\delta( w  )}{2}.\end{equation}

When $|z-w|\ge \delta(w)$, we have $\epsilon =  2^{-1}{\delta(w)} $. Thus from (\ref{G}) and (\ref{comp}),  
$$
\frac{1}{ \pi\epsilon}\left | \int_0^{2\pi}G(w+\epsilon e^{it}, z)dt   \right | \le \frac{6 C }{|z-w|}  \log \frac{2d}{|z-w|}.
$$
When $|z-w|\le \delta(w)$, $\epsilon =  2^{-1} {|z-w|} $. Since  $|G(w, z)| \le  \frac{1}{2\pi}\log\frac{d}{|z-w|}$  (see \eqref{G<} for instance), we have
$$
\frac{1}{\pi \epsilon} \left | \int_0^{2\pi}G(w+\epsilon e^{it}, z)dt  \right | \le \frac{2\pi ^{-1}}{|z-w|}  \log \frac{2d}{|z-w|}.
$$
Therefore, we get 
$$
|S(w, z)|\le \frac{ 2\max \{\pi ^{-1}, 3 C \} }{|z-w|}   \log \frac{ 2d}{|z-w|}.
$$  
For the second inequality, it suffices to  make a similar argument replacing  (\ref{G}) and \eqref{G<} by (\ref{g2}) and (\ref{G}), respectively. 

\end{proof}

 \vspace{0.15cm}
 
For a bounded domain  $\Omega$  in $\mathbb C^n$, let $A^2(\Omega) := L^2(\Omega) \cap \hbox{ker}(\bar \partial)$ represent the space of square integrable holomorphic functions on  $\Omega$.  The Bergman projection  on $\Omega$, denoted by  $\mathbf P$, is the orthogonal projection of $L^2(\Omega)$ onto its closed subspace $A^2(\Omega)$. Let $K$ be the Bergman kernel such that for all $h \in L^2(\Omega)$ and $w \in \Omega$,
\begin{equation} \label{reproducing}
\mathbf P h(w) = \int_{\Omega} K(w, z) h(z) {d \nu_z},
\end{equation}
where $d\nu$ is the Lebesgue $\mathbb{R}^{2n}$ measure.
A  solution $u \in L^2(\Omega)$ to the $\bar\partial$-equation is  canonical  if $\mathbf P u =0$. 

On a planar domain $D$, recall (see \cite{BL91, Bell}) that for $(z, w) \in D \times D $ off the diagonal, the Bergman kernel is related to Green's function $G$ by
\begin{equation*}
K(z, w)= {-4} \frac{\partial^2 G(z, w)} {\partial z \partial \bar w}.
\end{equation*}
Consequently, for $L(w, z)$ expressed in \eqref {L=H+I} over $D$  with $C^2$ boundary, it holds that
\begin{equation}\label{LtoK}
 K(w, z) = \frac{ - i}{2\pi} \int_{\zeta\in b D}  \frac{K(\zeta, z) }{\zeta-w} d\zeta= \frac{2 i}{\pi} \int_{\zeta\in b D} \frac{\partial^2  G(z, \zeta)}{\partial \zeta \partial \bar z} \frac{1 }{\zeta-w} d\zeta= {2 i} \frac{\partial L(w, z)}{\partial \bar z}.
 \end{equation}

 Barletta and Landucci observed that $\mathbf T$ defined in \eqref{c1} gives the canonical solution to the $\bar\partial$-equation, and its uniform estimate was written as a remark in \cite[p. 103] {BL91} without proof. For the data in the Lebesgue function spaces,  the result below is a consequence of Proposition \ref{|L|}. For abbreviation of  notations, write $\|\cdot\|_p: = \|\cdot\|_{L^p }$ for $p \in [1, \infty]$ here and henceforth, when there is no confusion of domains.

\begin{pro} \label{dim 1} Let $D$ be a  bounded domain with $C^2$ boundary in $\mathbb C$. Then for any $f \in L^p(D),$ $p > 1$, it holds that $\mathbf Tf\in L^2(D)$ and $\mathbf T$ is the canonical solution operator for $\bar\partial u =fd\bar z$. Moreover,  
\begin{enumerate}[label=(\roman*)]

\item for any $ p\in [1,  2]$ and  $q \in (0,  \frac{2p}{2-p})$, there exists a positive constant $C=C(D, p, q)$ such that for any $f\in L^p(D) $, 
\begin{equation}\label{pq}
    \|\mathbf Tf\|_{q}\le C\|f\|_{p };
    \end{equation}
    
\item for any $p \in (2, \infty]$, there exists a positive constant $C=C(D, p)$ such that for any $f \in L^p (D)$,   \begin{equation}\label{pq2}
    \|\mathbf Tf\|_{\infty}\le C\|f\|_{p}.
    \end{equation}
    
\end{enumerate}
\end{pro}

\begin{proof}
By the first inequality in (\ref{S}), the kernel $S$ is bounded by a convolution type function in $ L^r(D)$ for any $r \in [1, 2)$. Young's convolution inequality then implies the boundedness of $\mathbf T$ from $L^p(D)$ into $L^q(D)$ whenever $p^{-1}+ r^{-1} = q^{-1}+1$ with $  p, q, r \in [1, \infty]$. (\ref{pq}) and (\ref{pq2}) thus follow.

It remains to show that the solution operator $\mathbf T$ is canonical.
For any $f \in L^p(D),$ $p > 1$, notice by (i) that  $\mathbf Tf\in L^2(D)$ so $\mathbf P\mathbf T f $ is well defined. By the Cauchy--Pompeiu formula, \eqref{Dirichlet} and \eqref{LtoK},
\begin{align*}
    \mathbf P \mathbf T f  (\zeta) 
        =&\int_D { f} (z) \left(\int_D K(\zeta, w)L(w, z) d\nu_w + \frac{1}{2\pi i}\int_D \frac{K(\zeta, w)}{w-z}d\nu_w\right) d\bar z\wedge dz\\
    =&\int_D {  f} (z) \left(L(\zeta, z) -\frac{1}{2\pi i}\int_D \frac{\partial L(\zeta, w)}{\partial \bar w} \frac{1}{w-z} {dw\wedge d \bar w}  \right )  d\bar z\wedge dz\\
    =&\int_D {  f} (z)\left ( \frac{1}{2\pi i} \int_{bD} \frac{ L(\zeta, w)}{w-z} dw \right)  d\bar z\wedge dz\\
    =& \int_D {  f} (z)  \frac{1}{(2\pi i)^2} \frac{1}{z- \zeta}  \left ( \int_{bD}  \frac{ dw}{w-z } - \int_{bD}  \frac{ dw}{w-\zeta}  \right)  d\bar z\wedge dz =0,
\end{align*}
which means that $(\mathbf T f  ) \perp A^2(D)$ and the proof is complete.

\end{proof}

As a side product of Proposition \ref{dim 1}, the corollary below on the Bergman projection follows immediately from the fact that for any $u \in L^2(D) \cap  \text{Dom}(\bar\partial)$,
\begin{equation*}
  \mathbf Pu=u- \mathbf T \bar\partial u.
   \end{equation*}
 
\begin{cor} \label{cor=dim1}
Let $\mathbf P$ be the Bergman projection over a  bounded domain $D$ with $C^2$ boundary in $\mathbb C$. Then there exists a positive constant $C$ depending only on $D$ such that for any $u\in W^{1, \infty}(D)$ 
it holds that
$$
    \|u- \mathbf Pu\|_{ \infty  }\le  C  \|   \bar\partial u  \|_{ \infty }  ;   \quad \quad
  \|\mathbf Pu\|_{ \infty  }\le   \|u\|_{  \infty  }  + C\|   \bar\partial u  \|_{ \infty }.
$$

\end{cor}

 \section{Derivatives of one dimensional canonical solution kernel}

\vspace{0.15cm}

Let $D \subset \mathbb C$ be a bounded domain with $C^{2}$  boundary. The goal of the section is to estimate the gradient of the canonical solution kernel $S$ on $D$. In view of  \eqref{sl}, we shall need bounds for the derivatives of Green's function as follows. Here  $\nabla G$ represents the first order derivatives of $G$ with respect to either $z$ or $w$ variable, while $\nabla_z G$ represents the first order derivatives of $G$ with respect to $z$ variable.

\begin{lem}  \label{Kerzman-Green2} Under the same assumptions as in Lemma \ref{Kerzman-Green}, it holds that 
 \begin{equation}\label{Gd}
     |\nabla G( w, z)|\le \frac{C}{|z-w|}\log\frac{2d}{|z-w|}; \quad \quad   |\nabla_z G(w, z)| \le \frac{C\delta(w)}{|z-w|^2}\log \frac{2d}{|z-w|},
    \end{equation}
    for all $(w, z) \in  D \times \bar D \setminus \{z=w\}$.
\end{lem}

 Lemma \ref{Kerzman-Green2} is essentially an application of the mean value theorem for harmonic functions to Lemma \ref{Kerzman-Green}, and we will give its proof in Appendix \ref{AppB}.  Such  estimates for real dimensions bigger than two can be found in the comment after Theorem 1.2.8 in \cite{Kenig} without proof. Making use of a similar approach as in  Proposition \ref {|L|}, we obtain the following estimate for the derivatives  of the canonical solution kernel.

\begin{pro}\label{kkd}
Under the same assumptions as in Proposition \ref{|L|}, it holds that 
\begin{equation*}
|\nabla_z S(w, z) |\le \frac{C }{|z-w|^{2}}  \log \frac{2d}{|z-w|}.
\end{equation*}
\end{pro}

\begin{proof}
Fix $w\in D$ and let $z_0  \in \bar D \setminus \{w\}$. By \eqref{sl},

\begin{equation} \label{K-est}
  \nabla_z S(w, z_0 )   =   \left. \nabla_z  \left(\frac{i}{\epsilon\pi    }\int_0^{2\pi}G(w+\epsilon e^{it}, z)e^{-it}dt \right) \right |_{z=z_0},
  \end{equation}
where $\epsilon$ can be any positive number no bigger than $
\epsilon (w, z):=2^{-1}{\min} \left \{ {|z-w|} , {\delta(w)}  \right \}>0$.  
Set $\epsilon_0: = 2^{-1}\epsilon (w, z_0) = 4^{-1}{\min} \left \{ {|z_0-w|} , {\delta(w)}  \right \}$. 
For any point $z$ sufficiently close to $z_0$, the continuity of $\epsilon( w, z)$ in $z$ implies that $\epsilon_0\le \epsilon (w, z) $. Then on the right hand side of \eqref{K-est}, $\epsilon$   can be replaced by $\epsilon_0$, as $z$ approaches $z_0$. So
\begin{equation*}
  \nabla_z S(w, z_0 ) =  \frac{i}{\epsilon_0 \pi }  \int_0^{2\pi}  \left. \nabla_z     G(w+ \epsilon_0 e^{it}, z)\right |_{z=z_0} e^{-it}dt.
\end{equation*}
When $\epsilon_0  =   4^{-1} \delta(w) $, by the second part of \eqref{Gd} and (\ref{comp}),  we have
$$ 
|\nabla_z S(w, z_0 )|  \le   \frac{ 4}{ \delta(w) \pi  } \left| \int_0^{2\pi}   \left.  \nabla_z     G(w+\epsilon_0  e^{it}, z)\right |_{z=z_0}  e^{-it}dt \right| \le     \frac{ 48C  }{|z_0-w|^2}\log \frac{2d}{|z_0-w|};
$$
when $\epsilon_0  = 4^{-1} {|z_0-w|}$, the first part of \eqref{Gd} gives the desired estimate
$$
 |\nabla_z S(w, z_0 )|  \le  \frac{32C }{   {|z_0-w|^2}    }\log \frac{2d}{|z_0-w|}.
$$
Since $z_0$ is arbitrary, the proof is complete.

\end{proof}

Recall that the Bergman kernel $K$ is  associated with a derivative of the canonical solution kernel $S$ in terms of (\ref{LtoK}).  As a consequence of this and Proposition \ref{kkd}, we have an estimate of $K$ as follows.  

\begin{cor}
Under the same assumptions as in Proposition \ref{|L|}, it holds that 
$$
|K (w, z)|\le \frac{C }{|z-w|^{2}}  \log \frac{2d}{|z-w|}.
$$
\end{cor}
 
\medskip

 The stability for the  Bergman kernel  is already known in literature. See \cite{Boas, CZ} for instance.
Below we study the stability of the canonical solution kernel $S$ and its gradient under the mild variation of the underlying domains. The following locally uniform convergence of $S$ is proved using PDEs theory and plays a key role in reducing the regularity of the data to continuity in Section 5.

\begin{pro}\label{DD}
Let $D$ be a bounded domain in $\mathbb C$ with $C^2$ boundary. Let $\{D^l\}_{l=1}^\infty$ be an exhausting family of open subsets of $D$ with $C^2$ boundaries such that \\
a) $\bar D^l\subset D^{l+1}$;\\
 b) $h^l: \bar D\rightarrow \bar D^l$ is a $C^2$ diffeomorphism with  $\lim \limits_{l\rightarrow \infty}\|h^l-id\|_{C^2(D)}= 0$.\\
Let  $S$ and $S^l$ be the kernels of the canonical solution operators for the $\bar\partial$-equations on $D$ and $D^l$, respectively.
Then for each compact subset $K \Subset   D$, 
$S^l(w, h^l(z))$ and $\nabla S^l(w, h^l(z)) $ converge uniformly on $K\times \bar D \setminus  \{z=w\}$ to $S(w, z)$ and $\nabla S(w, z) $, respectively.
\end{pro}

\begin{proof}
Assume $K\Subset D^l$ for all $l$, and we will first prove that 
\begin{equation}\label{stable}
 \lim \limits_{l\rightarrow \infty}\sup \limits_{(w, z)\in K\times \bar D} \left  |S^l(w, h^l(z)) - S(w, z) \right | = \, 0.
\end{equation} 
Let $L_w$ and $ L^l_w $ be the solutions to \eqref{Dirichlet} on $D$ and $D^l$, respectively. 
Since both $L_w$ and $L^l_w$ are harmonic on $D^l$, by the Maximum Principle we know that for any $w\in K$,
\begin{equation}\label{DD6}
\sup_{z\in \bar D}|L^l_w(h^l(z))-L_w(h^l(z))| = \sup_{\zeta\in \bar D^l}|L^l_w(\zeta)-L_w(\zeta)|\le \sup_{\zeta\in bD^l}|{H(w, \zeta)-L_w(\zeta)}|. \end{equation}
Since $w$ belongs to $K$,  the universal Cauchy kernel $H (w, \cdot)\in C^\infty(bD)$ and has a uniform $C^2(bD)$ norm that is independent of $w$. By the $L^{p}$ theory in PDE, $L_w\in W^{2, p}(D)$ for all $p>1$ with $\|L_w\|_{W^{2, p}(D)}\le C$ independent of $w$. Thus $L_w\in C^{1, \alpha}(\bar D)$ for all $0<\alpha<1$ by the Sobolev embedding theorem with 
\begin{equation}\label{lll}
    \|L_w\|_{C^{1, \alpha}(D)}\le C
\end{equation}  independent of $w$.  Next, for each   $\zeta\in bD^l$, write $\zeta = h^l(z)$  for some $z\in bD$.  In view of the boundary condition on $L_w$, we know that 
 $$
    |H(w,\zeta )-L_w(\zeta)|   
        \le    |H(w, h^l(z))-H(w, z)|+  |L_w(z)-L_w(h^l(z))|  \le   (C +\|L_w\|_{C^{\alpha}(D)})| h^l(z)-z|^\alpha. 
    $$
By (\ref{lll}) and the assumption of $h^l$, one infers that
\begin{equation*}
    \sup_{(w, \zeta)\in  K\times bD^l}  |H(w, \zeta)-L_w(\zeta)|\rightarrow 0,
\end{equation*}
as $l\rightarrow \infty$. Furthermore, by (\ref{DD6}) it follows that
$$
    \sup_{(w,z)\in K\times \bar D}|L^l_w(h^l(z))-L_w(h^l(z))| \rightarrow 0,
 $$
which yields \eqref{stable} from  (\ref{sl}).

Adapt the previous argument for $H$, $L_w$ and $ L_w^l$ to $\nabla H$, $\nabla L_w$ and $\nabla L_w^l$, respectively. One proves an identity similar to \eqref{stable} for $\nabla S$ by virtue of the fact that $\nabla L_w\in C^\alpha(D)$ and thus completes the proof.
 
\end{proof}

\section{Proof of the main theorem under differentiability}
 
In this section we prove the main Theorem \ref{main} on uniform estimates for the canonical solution to $\bar \partial u =f$, under the assumption that $f$ is $C^{n-1}$ up to the boundary of product domains. More precisely, we shall prove

 \begin{theorem} \label{C1}
Let $\Omega:=D_1\times \cdots \times D_n$, $n\geq 2$, where each $D_j$ is a bounded planar domain with $C^2$ boundary. Then there exists a positive constant $C$ depending only on $\Omega$ such that for any $\bar\partial$-closed form $ f \in C^{n-1}_{(0,1)} (\bar\Omega)$,  the canonical solution to  $\bar\partial u = f$ satisfies 
\begin{equation*}
\|  u   \|_{ \infty }\le C\|  f\|_{ \infty },
\end{equation*}
where $\| \cdot  \|_{\infty}$ is the essential supnorm on $\Omega$.
\end{theorem}

 In this first subsection, we study the canonical solution to the $\bar \partial$ equation on the Cartesian product of planar domains, and construct the canonical solution operator, inspired by the recent work \cite{FP19}. The canonical solution operator for planar domains is given together with uniform estimates in Section 2. Since the case of $n=2$ carries the precise idea yet without involving too much technical computations,  we prove Theorem \ref{C1} for this case in the second subsection. After that, we deal with the general case of arbitrary $n$ in the third subsection.  
 
\vspace{0.15cm}

The crucial idea in the proof is to eliminate the derivative terms in (\ref{formula1}) by the $\bar\partial$-closeness of the data, making use of a similar  idea as in \cite{FP19, L75}. In the last two subsections, after conducting  a refined decomposition of the solution kernel and eliminating the derivatives of the data in (\ref{formula1}), we shall verify each decomposed portion still maintains the integrability. The following elementary geometric mean inequality  serves the role well. 

\begin{lem} \label{ele}
Let $x_j\ge 0$ and $  \alpha_j \in [0,  1]$, $1\le j\le n$. Then it holds that
$$
  \prod \limits_{j=1}^n x_j^{\alpha_j}   \le  
\begin{cases}
\quad   \sum \limits _{j=1}^n \alpha_j x_j &  \text{whenever \,}   \sum\limits_{j=1}^n\alpha_j =1; \\ 
 \quad \sum \limits_{j=1}^n(1-\alpha_j)\prod \limits_{m \ne j, m=1}^n x_m &  \text{whenever \,}   \sum\limits_{j=1}^n\alpha_j =n-1. 
\end{cases}
 $$
\end{lem}

\begin{proof}
The first inequality is the standard weighted geometric mean inequality.   The second inequality follows from the first one, together with the observation that 
$$\prod \limits_{j=1}^n x_j^{\alpha_j} =  \prod \limits_{j=1}^n \left(\prod \limits_{m \ne j, m=1}^n x_m\right)^{1-\alpha_j}.$$
\end{proof}

For the rest of the paper, we use $C$ to represent a constant depending only on each domain. It may be different at various contexts.

\subsection{Canonical solution formula on the product of planar domains}

A (non-canonical) solution operator, which consists of compositions of the solid Cauchy integral along each slice, first appeared in \cite{FP19}. It is our observation that, to construct the canonical solution operator on $\Omega$, one only needs to replace each of slice-wise Cauchy operators by the corresponding  canonical solution operator.

\begin{theorem} \label{diff}
Let $\mathbf f=\sum \limits_{j=1}^nf_jd\bar z_j$ be a $\bar\partial$-closed $(0,1)$ form on $\Omega$ such that
$f_j  \in C^{n-1} (\bar\Omega)$.
 Then the canonical solution $\mathbf T\mathbf f$ to  $\bar\partial u =\mathbf f$ is represented as
\begin{equation}\label{formula1}
  \mathbf T\mathbf f: =  \sum \limits_{s=1}^{n} (-1)^{s-1}\sum \limits_{1\le i_1<\cdots< i_s\le n} \mathbf T_{i_1}\cdots \mathbf T_{i_s}(\frac{\partial^{s-1} f_{i_s}}{\partial \bar z_{i_1}\cdots \partial\bar z_{i_{s-1}}}), 
\end{equation}
where $\mathbf  T_j$ is the canonical solution operator over $D_j$. 
\end{theorem}

\begin{proof} Firstly, the operator $\mathbf T$ in (\ref{formula1}) is well defined over $C^{n-1} (\bar\Omega)$, with its supnorm bounded by the $C^{n-1} (\bar\Omega)$ norm of the data in view of the estimate in Proposition \ref{dim 1}.
     
     \medskip
     
Secondly, denote by $\mathbf P_j$ and $\mathbf  P$  the Bergman projection operator over $D_j$ and  $\Omega$, respectively. Then
  $ \mathbf  P= \mathbf P_1 \cdots \mathbf P_n $  and each term in \eqref{formula1} contains some $\mathbf T_j$ while $\mathbf P_j\mathbf T_j=0$ by the proof of Proposition \ref{dim 1}. Repeated application of Fubini's theorem implies that $\mathbf P\mathbf T\mathbf f=0$, i.e., $$ \mathbf T f  \perp A^2(\Omega).$$

Thirdly, we show that $\bar\partial \mathbf T\mathbf f = \mathbf f$ using a direct computation, similar to \cite{FP19}. The $\bar\partial$-closeness of $\mathbf f$ implies that (\ref{formula1}) is symmetric with respect to the roles of $z_j$ and $z_n$, so we only need to prove $\bar\partial_n \mathbf T\mathbf f = f_n$. Isolate terms containing $\mathbf T_n$ in (\ref{formula1}) and rewrite
\begin{equation*}
 \begin{split} \mathbf T\mathbf f= & \sum_{s=1}^{n-1} (-1)^{s-1}\sum_{1\le i_1<\cdots< i_s\le n-1} \mathbf T_{i_1}\cdots \mathbf T_{i_s}(\frac{\partial^{s-1} f_{i_s}}{\partial \bar z_{i_1}\cdots \partial\bar z_{i_{s-1}}}) \, +\\ &+ \mathbf T_n f_n +  \sum_{s=2}^{n} (-1)^{s-1}\sum_{1\le i_1<\cdots< i_{s-1}\le n-1, i_s=n} \mathbf T_{i_1}\cdots \mathbf T_{i_{s-1}}\mathbf T_{n}(\frac{\partial^{s-1} f_{n}}{\partial \bar z_{i_1}\cdots \partial\bar z_{i_{s-1}}}).
 \end{split}
\end{equation*}
Applying $\bar\partial_n$ to each term in the above expression, we obtain by $\bar\partial_n \mathbf T_n =id$ that
\begin{equation*}
    \begin{split}
  \bar\partial_n \mathbf T\mathbf f =&    \sum_{s=1}^{n-1} (-1)^{s-1}\sum_{1\le i_1<\cdots< i_s\le n-1} \mathbf T_{i_1}\cdots \mathbf T_{i_s}(\frac{\partial^{s} f_{i_s}}{\partial \bar z_{i_1}\cdots \partial\bar z_{i_{s-1}}\partial\bar z_n}) \\ &+ f_n +  \sum_{s=2}^{n} (-1)^{s-1}\sum_{1\le i_1<\cdots< i_{s-1}\le n-1} \mathbf T_{i_1}\cdots \mathbf T_{i_{s-1}}(\frac{\partial^{s-1} f_{n}}{\partial \bar z_{i_1}\cdots \partial\bar z_{i_{s-1}}}).
    \end{split}
\end{equation*}

Lastly, since  $\mathbf f$ is $\bar\partial$-closed and
$$
\frac{\partial^{s} f_{i_s}}{\partial \bar z_{i_1}\cdots \partial\bar z_{i_{s-1}}\partial\bar z_n} = \frac{\partial^{s} f_{n}}{\partial \bar z_{i_1}\cdots \partial\bar z_{i_{s-1}}\partial\bar z_{i_s}},
$$
 we replace $s$   by $s-1$ in the last summation above to get 
\begin{align*}
  \bar\partial_n \mathbf T\mathbf f = &    \sum_{s=1}^{n-1} (-1)^{s-1}\sum_{1\le i_1<\cdots< i_s\le n-1} \mathbf T_{i_1}\cdots \mathbf T_{i_s}(\frac{\partial^{s} f_{n}}{\partial \bar z_{i_1}\cdots \partial\bar z_{i_{s-1}}\partial\bar z_{i_s}})\\ 
  &+ f_n+  \sum_{s=1}^{n-1} (-1)^{s}\sum_{1\le i_1<\cdots< i_{s}\le n-1} \mathbf T_{i_1}\cdots \mathbf T_{i_{s}}(\frac{\partial^{s} f_{n}}{\partial \bar z_{i_1}\cdots \partial\bar z_{i_{s}}}) =  f_n.
\end{align*}

\end{proof}

\subsection{Uniform estimate for $n=2$}
 
When $n=2$, $\mathbf T \mathbf f = \mathbf T_1f_1+\mathbf T_2f_2-\mathbf T_1\mathbf T_2(\mathcal Df)$, where $\mathcal Df = \frac{\partial f_1}{\partial \bar z_2} = \frac{\partial f_2}{\partial \bar z_1}$. By Proposition \ref{dim 1}, we only need to estimate $\|\mathbf T_1\mathbf T_2(\mathcal Df)\|_{\infty}$. Let $|{w}-z|^2: = |{w}_1-z_1|^2+|{w}_2-z_2|^2$  and $ dV_z:=d\bar z_1\wedge dz_1\wedge d\bar  z_2\wedge dz_2$. Let $S_j$ be the  canonical solution kernel  on $D_j$ defined in (\ref{sl}) for $j=1, 2$.

\begin{proof} [Proof of Theorem \ref{C1} ($n=2$)]

Making use of  the trivial identity
\begin{equation*}
    S_1 S_2  =  \frac{S_1 S_2 |{w}_1-z_1|^2}{|{w}- z|^2} +\frac{S_1 S_2 |{w}_2-z_2|^2}{|{w}- z|^2},
    \end{equation*}
    we rewrite $ \mathbf T_1\mathbf T_2(\mathcal Df)$ as 
\begin{equation*}
\begin{split}
 \mathbf T_1\mathbf T_2(\mathcal Df)({w})=
&\int_{D_1 \times D_2}{\mathcal D (f)(z_1, z_2)}{S_1({w}_1, z_1)S_2({w}_2,z_2)}dV_z\\
=&\int_{D_1 \times D_2}{\frac{\partial f_2(z)}{\partial \bar z_1}}\frac{S_1({w}_1, z_1)S_2({w}_2,z_2)|{w}_1-z_1|^2}{|{w}- z|^2}dV_z\\
&+\int_{D_1 \times D_2}{\frac{\partial f_1(z)}{\partial \bar z_2}}\frac{S_1({w}_1, z_1)S_2({w}_2,z_2)|{w}_2-z_2|^2}{|{w}- z|^2}dV_z =: I_1+I_2.
\end{split}
\end{equation*}
We shall only estimate  $\|I_2\|_{\infty}$, since $\|I_1\|_{\infty}$ is handled similarly due to symmetry.

     \medskip
     
We assert that
\begin{equation}\label{stok1}
          \begin{split}
          I_2 =&\int \limits_{{  D}_1} \int \limits_{b{  D}_2} \frac{f_1(z) S_1({w}_1, z_1)S_2({w}_2,z_2)|{w}_2-z_2|^2}{|{w} - z|^2} dz_2\wedge d\bar z_1\wedge dz_1  -\int \limits_{D_1\times   D_2} f_1( z)S_1({w}_1, z_1)  \\
          &
 \cdot \Big( \frac{\partial S_2({w}_2, z_2)}{\partial \bar z_2} \frac{|{w}_2- z_2|^2}{|{w}-z|^2} - \frac{ S_2({w}_2, z_2)({w}_2- z_2)|{w}_1-z_1|^2}{|{w}- z|^4}\Big)dV_z.\end{split}
   \end{equation}

To see this, we write $  R^\epsilon_2: =   D_2 \setminus B({w}_2, \epsilon)$, where $B({w}_2, \epsilon)$ is a disc centered at ${w}_2$ with radius $\epsilon<<1$. Then by Stokes' Theorem,
\begin{equation} \label{stok}
    \begin{split}
   &\int_{D_1\times   R^\epsilon_2}\frac{\partial f_1( z)}{\partial \bar z_2}  \frac{S_1({w}_1, z_1)S_2({w}_2,z_2)|{w}_2-z_2|^2}{|{w}- z|^2} dV_z \\
   =&  \int_{{  D}_1} \int_{b{  D}_2}  \frac{f_1(z) S_1({w}_1, z_1)S_2({w}_2,z_2)|{w}_2-z_2|^2}{|{w} - z|^2}dz_2\wedge d\bar z_1\wedge dz_1 \\
   &- \int_{{  D}_1}\int_{b{B({w}_2, \epsilon)}}  \frac{f_1(z) S_1({w}_1, z_1)S_2({w}_2,z_2)|{w}_2-z_2|^2}{|{w} - z|^2} dz_2\wedge d\bar z_1\wedge dz_1\\
&-\int \limits_{{  D}_1\times R^\epsilon_2} f_1( z)S_1({w}_1, z_1) \Big( {\frac{\partial S_2({w}_2, z_2)}{\partial \bar z_2} }\frac{|{w}_2- z_2|^2}{|{w}-z|^2} - \frac{ S_2({w}_2, z_2)({w}_2- z_2)|{w}_1-z_1|^2}{|{w}- z|^4}\Big)dV_z.    \end{split}
\end{equation}
Notice that by (\ref{S}),  for any $\sigma>0$, for $ (z_1, z_2)\in (D_1\setminus \{ z_1 = {w}_1\}) \times ( D_2\setminus \{ z_2 = {w}_2\})$,
$$ \frac{|S_1({w}_1, z_1) S_2({w}_2,z_2)||{w}_2-z_2|^2}{|{w}- z|^2}\le \frac{C|{w}_2-z_2|^{1-\sigma}}{|{w}_1-z_1|^{1+\sigma}|{w}-z|^{2}}\le  \frac{C}{|{w}_1-z_1|^{1+\sigma}|{w}_2-z_2|^{1+\sigma}}, $$
which belongs to $ L^1(D_1\times D_2)$ by choosing   $\sigma$ less than 1. 
Because $f_1\in C^{1}(\bar\Omega)$,  the left hand side of (\ref{stok}) approaches $ I_2  $ as $\epsilon\rightarrow 0$ by Lebesgue's Dominated Convergence Theorem.

 On the other hand, in Lemma \ref{ele} choosing $\alpha_1=1/4$, $\alpha_2=3/4$ and $x_j=|{w}_j-z_j|$, for $j=1,2$, we know that 
 $|{w}-z|^2\ge C |{w}_1-z_1|^{\frac{1}{2}}|{w}_2-z_2|^{\frac{3}{2}}$. So for any $\sigma>0$, 
 \begin{equation} \label{n1}\frac{|S_1({w}_1, z_1) S_2({w}_2,z_2)||{w}_2-z_2|^2}{|{w}- z|^2}\le    \frac{C}{|{w}_1-z_1|^{\frac{3}{2}+\sigma}|{w}_2-z_2|^{\frac{1}{2}+\sigma}}\in L^1_{loc}(\mathbb C_{z_1}\times \mathbb R_{z_2})  \end{equation}
when choosing $\sigma$ small (say, less than $\frac{1}{2}$). The second term on the right hand side of (\ref{stok}) satisfies
 \begin{equation*}
 \begin{split}
     \bigg|\int \limits_{{  D}_1}  \int\limits_{b{B({w}_2, \epsilon)}}  \frac{f_1(z) S_1({w}_1, z_1)S_2({w}_2,z_2)|{w}_2-z_2|^2}{|{w} - z|^2} dz_2\wedge d\bar z_1\wedge dz_1\bigg|
\le \int_{b{B({w}_2, \epsilon)}} \frac{C|dz_2|}{|{w}_2-z_2|^{\frac{1}{2}+\sigma}} \rightarrow 0,
\end{split}
 \end{equation*} 
 as $\epsilon\rightarrow 0$.
For the third term on the right hand side of (\ref{stok}), by the fact that $$|{w}-z|^2\ge C|{w}_1-z_1|^{\frac{1}{2}}|{w}_2-z_2|^{\frac{3}{2}},$$
we know from (\ref{S}) and Proposition \ref{kkd} that there exists a small $\sigma$ such that
 \begin{equation}\label{n2} \begin{split}
\left|S_1({w}_1, z_1)  \frac{\partial S_2({w}_2, z_2)}{\partial \bar z_2} \right| \frac{|{w}_2- z_2|^2}{|{w}-z|^2}     & \le  \frac{C}{|{w}_1-z_1|^{1+\sigma}|{w}_2-z_2|^{\sigma}|{w}-z|^{2}}\\
 &  \le  \frac{C}{|({w}_1-z_1) ({w}_2-z_2)|^{\frac{3}{2}+\sigma}}\in L^1(D_1\times D_2). 
 \end{split}
 \end{equation}
 Similarly, using $|{w}-z|^4\ge C |{w}_1-z_1|^{\frac{5}{2}}|{w}_2-z_2|^\frac{3}{2}$ and letting $\sigma$ less than $\frac{1}{2}$, we get
 \begin{equation}\label{n3}\begin{split}\frac{ |S_1({w}_1, z_1) S_2({w}_2, z_2) ({w}_2- z_2)| |{w}_1-z_1|^2}{|{w}- z|^4}
 \le& \frac{C|{w}_1-z_1|^{1-\sigma}}{|{w}_2-z_2|^{\sigma}|{w}-z|^4}\\
 \le & \frac{C}{|({w}_1 - z_1) ( {w}_2 - z_2)|^{\frac{3}{2}+\sigma}}\in L^1(D_1\times D_2).\end{split} \end{equation}
As $\epsilon\rightarrow 0$, the last line on the right hand side of (\ref{stok}) approaches $$\int_  {D_1\times D_2} f_1( z)S_1({w}_1, z_1)\Big( {\frac{\partial S_2({w}_2, z_2)}{\partial \bar z_2}  } \frac{|{w}_2- z_2|^2}{|{w}-z|^2} - \frac{ S_2({w}_2, z_2)({w}_2- z_2)|{w}_1-z_1|^2}{|{w}- z|^4}\Big)dV_z,$$
from which the assertion (\ref{stok1}) follows immediately. Furthermore, by the proof of the assertion, in particular by (\ref{n1}-\ref{n3}), we obtain
 \begin{equation*}
     \|I_2\|_{\infty}\le C \|\mathbf f\|_{\infty},
 \end{equation*}
so the proof of Theorem \ref{C1} is complete for $n=2$, when  the data are $C^1$ up to the boundary.  

\end{proof} 

\subsection{Uniform estimate for arbitrary $n$}

For the rest of the paper, for convenience, we will suppress the corresponding measure element from a given integral. We prove by induction on $n$ that for all $1\le i_1<\cdots< i_s\le n$,
\begin{equation} \label{induction}
    \big\|\mathbf T_{i_1}\cdots \mathbf T_{i_s} \left (\frac{\partial^{s-1} f_{i_s}}{\partial \bar z_{i_1}\cdots \partial\bar z_{i_{s-1}}} \right )\big\|_{\infty}\le C \|\mathbf f\|_{\infty}.
\end{equation}
 Assuming \eqref{induction} is true for $n=k-1, k\ge 3$, we  shall verify that \begin{equation}\label{indu1}
  \|\mathbf T_1\cdots \mathbf T_k (\mathcal D^{k-1} \mathbf f)\|_{\infty} =  \left \| \int  _{  D_1 \times\cdots\times   D_k}\mathcal D^{k-1} (\mathbf f)(z)\prod_{j=1}^k S_j({w}_j, z_j) \right \|_{\infty}\le C \|\mathbf f\|_{\infty},
\end{equation}
where $\mathcal D^{k-1} \mathbf f := \frac{\partial^{k-1} f_1}{\partial \bar z_2\cdots  \partial \bar z_k} = \frac{\partial^{k-1} f_2}{\partial \bar z_1\partial \bar z_3\cdots  \partial \bar z_k} \cdots = \frac{\partial^{k-1} f_k}{\partial \bar z_1\cdots  \partial \bar z_{k-1}}$. The remaining cases are done by symmetry.
 Letting 
 \begin{equation} \label{H}
E : =  \sum_{j=1}^k \prod_{m=1, m\ne j}^{k} |{w}_m- z_m|^2,
 \end{equation}  
 we write
\begin{equation*}
   \mathcal D^{k-1} \mathbf f({w})\prod_{j=1}^k S_j({w}_j, z_j) = \sum_{j=1}^k E^{-1}{\prod_{l=1}^k S_l({w}_l, z_l)\prod_{m=1, m\ne j}^{k} } \frac{ |{w}_m- z_m|^2\partial^{k-1} f_j(z)}{\partial \bar z_1\cdots \partial \bar z_{j-1}\partial \bar z_{j+1} \cdots\partial \bar z_{k}}.
\end{equation*}
Due to the symmetry again, to prove (\ref{indu1}), it suffices to show
\begin{equation}\label{indu}
     \left\| \int_{  D_1 \times\cdots\times   D_k}  e({w}, z)\frac{\partial^{k-1} f_k({w})}{\partial \bar z_1\cdots \partial \bar z_{k-1}}  \right\|_{\infty}\le C \|\mathbf f\|_{\infty},
\end{equation}
where $$
    e({w}, z): = E^{-1}   {\prod_{l=1}^k S_l({w}_l, z_l)\prod_{m=1}^{k-1} |{w}_m- z_m|^2}.
$$

Before establishing estimates of the derivatives for the function $e({w}, z)$ defined above, the following lemma on $E$ defined in \eqref{H} will be needed.

\begin{lem}\label{ap}
For $ (z_1, \cdots, z_k)\in (D_1\setminus \{z_1={w}_1\}) \times\cdots\times   (D_k\setminus \{ z_k = {w}_k\})$, it holds that  for $m =1, \ldots, k-1 $, 
$$
\frac{\partial^m }{\partial \bar z_1\cdots\partial \bar z_m}  \left( \frac{ \prod_{j=1}^{k-1} |{w}_j- z_j|^2 }{E}\right) = \frac{m! |{w}_k-z_k|^{2m}}{E^{m+1}}  \prod_{j=1}^m \frac{{w}_j-z_j} {|{w}_j-z_j|^{2(1-m)} } \prod_{j=m+1}^{k-1}|{w}_j-z_j|^{2(m+1)}.
$$
\end{lem}

\begin{proof}
We prove it by induction on $m\in \{1, \ldots, k-1\}$. Since $m=1$ is trivial, assuming $m=l$ is true, by a straightforward computation  we know that when $m=l+1$,
 \begin{equation*}
     \begin{split}
         & \frac{\partial^{l+1} }{\partial \bar z_1\cdots\partial \bar z_l\partial \bar z_{l+1}} \left( \frac{ \prod_{j=1}^{k-1} |{w}_j- z_j|^2 }{E}\right)\\
         =&  \frac{\partial}{\partial \bar z_{l+1}} \left( \frac{l! |{w}_k-z_k|^{2l}}{E^{l+1}}  \prod_{j=1}^l \frac{{w}_j-z_j} {|{w}_j-z_j|^{2(1-l)} } \prod_{j=l+1}^{k-1}|{w}_j-z_j|^{2(l+1)}  \right)\\
          =& l! |{w}_k-z_k|^{2l}\prod_{j=1}^l \frac{{w}_j-z_j} {|{w}_j-z_j|^{2(1-l)} }  \prod_{j=l+2}^{k-1}|{w}_j-z_j|^{2(l+1)} \frac{\partial}{\partial \bar z_{l+1}} \left( \frac{ |{w}_{l+1}-z_{l+1}|^{2(l+1)}}{E^{l+1}}  \right)\\ 
          =& (l+1)! |{w}_k-z_k|^{2l}\prod_{j=1}^l \frac{{w}_j-z_j} {|{w}_j-z_j|^{2(1-l)} }  \prod_{j=l+2}^{k-1}|{w}_j-z_j|^{2(l+1)} \frac{(w_{l+1}-z_{l+1})|w_{l+1}-z_{l+1}|^{2l}\prod_{j=1, j\ne l+1}^k |w_j-z_j|^2}{E^{l+2}}\\ 
          =& {\frac{(l+1)! {|{w}_k-z_k|^{2(l+1)}}}{ E ^{l+2}} \prod_{j=1}^{l+1}\frac{  {w}_j-z_j }  {|{w}_j-z_j|^{-2l}} \prod_{j=l+2}^{k-1}   |{w}_j-z_j|^{2(l+2)} }. 
       \end{split}
 \end{equation*}

\end{proof}

We now estimate the derivatives of  $e(w, z)$ point-wisely, which is the key to the proof of (\ref{indu}).

\begin{pro}\label{ee} Let $\{i_1, \cdots, i_m, j_1, \cdots, j_{k-1-m}\}$ be a permutation of $\{1, \cdots, k-1\}$. Then for any $\sigma>0$, for $ (z_1, \cdots, z_k)\in (D_1\setminus \{z_1={w}_1\}) \times\cdots\times   (D_k\setminus \{ z_k = {w}_k\})$,
$$
\left|\frac{\partial^m  e({w}, z)}{\partial \bar  z_{i_1}\cdots\partial\bar z_{i_m}}\right|
     \le  {C} |{w}_k- z_k|^{-\frac{3}{2}- \sigma} { \prod_{l=1}^m |{w}_{i_l}- z_{i_l}|^{-2 +\frac{1}{2(k-1)}-\sigma}\prod_{l=1}^{k-1-m} |{w}_{j_l}-z_{j_l}|^{-1 + \frac{1}{2(k-1)}-\sigma} }.
$$
\end{pro}

\begin{proof}
Due to symmetry, we shall only estimate $\frac{\partial^m  e({w}, z)}{\partial \bar  z_{1}\cdots\partial\bar z_{m}}$, for $m\le k-1$. Namely, we prove
\begin{equation}\label{hh}
    \left|\frac{\partial^m  e({w}, z)}{\partial \bar  z_{1}\cdots\partial\bar z_{m}}\right|\le   C  |{w}_k- z_k|^{-\frac{3}{2}-\sigma}  {(\prod_{l=1}^m |{w}_{l}- z_{l}|^{-2 +\frac{1}{2(k-1)}-\sigma})(\prod_{l=m+1}^{k-1-m} |{w}_{l}-z_{l}|^{-1 + \frac{1}{2(k-1)}-\sigma}) }.
\end{equation}
Making use of Lemma \ref{ap}, (\ref{S}) and Proposition \ref{kkd}, one sees for any $\sigma>0$ that
\begin{equation}\label{hh1}\begin{split}
    \bigg|\frac{\partial^m  e({w}, z)}{\partial \bar  z_{1}\cdots\partial\bar z_{m}}\bigg| 
    \le &  \, C  \sum_{1\le k_1, \ldots, k_m\le m}  \bigg|\frac{\partial^{m-t} \prod\limits_{l=1}^k S_l({w}_l, z_l)}{\partial \bar  z_{k_{t+1}}\cdots\partial\bar z_{k_m}} \cdot \frac{\partial^{t} }{\partial \bar  z_{k_{1}}\cdots\partial\bar z_{k_t}}  \frac{\prod\limits_{m=1}^{k-1} |{w}_m- z_m|^2}{ E  } \bigg|\\
\le&  \, C \sum_{1\le k_1, \ldots, k_m\le m}\frac{  \prod\limits  _{l=1}^t |{w}_{k_l}-z_{k_l}|^{2t-2-\sigma} \prod \limits  _{l=t+1}^m|{w}_{k_l}-z_{k_l}|^{2t-\sigma} \prod\limits _{l=m+1}^{ k-1}|{w}_l-z_l|^{2t+1-\sigma}}{ E ^{ t+1}  |{w}_k-z_k|^{-2t+1+\sigma}},
    \end{split}
\end{equation}
where the sum is over all permutations of $(1, \ldots, m)$.
On the other hand, 
 we know from Lemma \ref{ele} that
$$
    E ^{ t+1} 
    \ge C \prod_{l=1}^t |{w}_{k_l}- z_{k_l}|^{2t -\frac{1}{2(k-1)}} \prod_{l=t+1}^m |{w}_{k_l}- z_{k_l}|^{2t +2-\frac{1}{2(k-1)}} \prod_{l=m+1}^{k-1} |{w}_{l} - z_{l}|^{2t+2 - \frac{1}{2(k-1)}}|{w}_k- z_k|^{2t+\frac{1}{2}},
$$
since both sides have the same total degree $2(k-1)(t+1)$.
 Combining the above inequality with (\ref{hh1}), we have proved (\ref{hh}) and thus the proposition.
 
\end{proof}

We are now ready to prove (\ref{indu}) and therefore give a complete proof of Theorem \ref{C1} for the arbitrary dimensional case,  when the data are $C^{n-1}$ up to the boundary.

\begin{proof} [Proof of Theorem \ref{C1} (arbitrary $n$)]
Similar to the case of $n=2$, denote by $R_j^\epsilon: =D_j\setminus B({w}_j, \epsilon)$ with $\epsilon$ small. Repeatedly applying  Stokes' Theorem, one obtains
\begin{equation}\label{st}
    \begin{split}
      &\int _{  R_1^\epsilon \times\cdots\times  R_{k-1}^\epsilon\times   D_k}  e({w}, z)\frac{\partial^{k-1} f_k( z)}{\partial \bar  z_1\cdots \partial \bar  z_{k-1}}  \\
    = & \int_{b   D_1\times   R_2^\epsilon \times\cdots\times  R_{k-1}^\epsilon\times   D_k } e({w}, z)\frac{\partial^{k-2} f_k( z)}{\partial \bar  z_2\cdots \partial \bar  z_{k-1}}  -  \int_{  R_1^\epsilon\times \cdots\times  R_{k-1}^\epsilon\times  D_k} \frac{\partial  e({w}, z)}{\partial \bar  z_1}\frac{\partial^{k-2} f_k( z)}{\partial \bar  z_2\cdots \partial \bar  z_{k-1}} \\
    & - \int_{ b B({w}_1, \epsilon)\times   R_2^\epsilon \times\cdots\times  R_{k-1}^\epsilon\times   D_k} e({w}, z)\frac{\partial^{k-2} f_k( z)}{\partial \bar  z_2\cdots \partial \bar  z_{k-1}}  \quad 
     = \quad \cdots\\
    =&\sum_{m=0}^{k-1} (-1)^m\sum_{1\le i_1<\cdots <i_j\le k-1} \int  _{  R^\epsilon_{i_1}\times \cdots  R^\epsilon_{i_m}\times b  D_{j_1}\times \cdots \times b  D_{j_{k-1-m}}\times   D_k}\frac{\partial^m  e({w}, z)}{\partial \bar  z_{i_1}\cdots\partial\bar z_{i_m}}f_k( z) \\
    &-\sum_{m=0}^{k-2} (-1)^m\sum_{1\le i_1<\cdots <i_j\le k-1} \sum_{l=1}^{k-1-m}\int \limits_{R_{\mathbf i\mathbf j, \epsilon}} \frac{\partial^m e({w}, z)}{\partial \bar  z_{i_1}\cdots\partial\bar z_{i_m}}f_k( z) ,
    \end{split}
\end{equation}
where $R_{\mathbf i\mathbf j, \epsilon}:= R^\epsilon_{i_1}\times \cdots  R^\epsilon_{i_m}\times b B({w}_{j_1}, \epsilon)\times \cdots \times b B({w}_{j_{l}}, \epsilon)\times b   D_{j_{l+1}}\times\cdots\times b   D_{j_{k-1-m}}\times   D_k$,
and $\{1, \cdots, k-1\}$ is a permutation of $\{i_1, \cdots, i_m, j_1, \cdots, j_{k-1-m}\}$. 
Letting $\epsilon\rightarrow 0$, we claim that (\ref{st}) reduces to
\begin{equation}\label{form}
\begin{split}
      \int \limits_{  D_1 \times\cdots\times     D_k} e({w}, z)\frac{\partial^{k-1} f_k( z)}{\partial \bar  z_1\cdots \partial \bar  z_{k-1}} 
       =\sum \limits_{m=0}^{k-1} (-1)^m &\sum _{1\le i_1<\cdots <i_j\le k-1}\\
       & \int \limits   _{  D_{i_1}\times \cdots \times D_{i_m}\times b  D_{j_1}\times \cdots \times b  D_{j_{k-1-m}}\times  D_k}\frac{\partial^m  e({w}, z)}{\partial \bar  z_{i_1}\cdots\partial\bar z_{i_m}}f_k( z).
       \end{split}
\end{equation}
To get (\ref{form}),   choose $\sigma = \frac{1}{4(k-1)}(<\frac{1}{4} )$ in Proposition \ref{ee}. Then with  $m=0$, one has
$$|e({w}, z)|\le C |{w}_k- z_k|^{-\frac{3}{2}-\frac{1}{4(k-1)}} \prod_{l=1}^{k-1} |{w}_{l}-z_{l}|^{-1 + \frac{1}{4(k-1)}} \in L^1( D_1 \times\cdots\times     D_k).$$
 So
the left hand side of (\ref{st}) as $\epsilon\rightarrow 0$ satisfies
\begin{equation*}
    \int_{  R_1^\epsilon \times\cdots\times  R_{k-1}^\epsilon\times   D_k}  e({w}, z)\frac{\partial^{k-1} f_k( z)}{\partial \bar  z_1\cdots \partial \bar  z_{k-1}} \rightarrow \int_{  D_1 \times\cdots\times   D_k} e({w}, z)\frac{\partial^{k-1} f_k( z)}{\partial \bar  z_1\cdots \partial \bar  z_{k-1}}.  
\end{equation*}
Similarly, for the last line of (\ref{st}),  by Proposition \ref{ee}, it holds that
 \begin{equation}\label{en}
 \begin{split}
    \left|\frac{\partial^m  e({w}, z)}{\partial \bar  z_{i_1}\cdots\partial\bar z_{i_m}}\right|&\le {C}|{w}_k- z_k|^{-\frac{3}{2}- \frac{1}{4(k-1)}} { \prod_{l=1}^m |{w}_{i_l}- z_{i_l}|^{-2 +\frac{1}{4(k-1)}}\prod_{l=1}^{k-1-m} |{w}_{j_l}-z_{j_l}|^{-1 + \frac{1}{4(k-1)}} }\\
    &\in L^1_{loc}(\mathbb C_{ z_{i_l}}\times\cdots \times\mathbb C_{ z_{i_m}} \times\mathbb R_{ z_{j_1}}\times \cdots\times \mathbb R_{ z_{j_{k-1-m}}}\times \mathbb C_{ z_k}). 
    \end{split}
 \end{equation}
As $\epsilon\rightarrow 0$,  the last line of (\ref{st}) vanishes similarly as the case of $n=2$. The formula (\ref{form}) thus holds. Since (\ref{indu}) follows from (\ref{form}) and (\ref{en}), we have finished the proof of Theorem \ref{C1} completely.

\end{proof}

\section{Proof of the main theorem under continuity}

Let $\Omega:= D_1\times\cdots\times D_n$, $n\ge 2$, where each $D_j$ is a bounded  domain with $C^2$ boundary in $\mathbb C$. 
We shall prove in this section the main Theorem \ref{main} by weakening the $C^{n-1}_{(0, 1)}(\bar\Omega)$ assumption in Theorem \ref{C1} to continuity.
In \cite{FP19}, an operator in terms of Cauchy integrals was proposed as a potential candidate to solve the $\bar\partial$-equation with continuous data. The following operator $\tilde {\mathbf T}$ is an analogue  but in the context of canonical solutions. For a given $\mathbf f \in C_{(0,1)}(\bar\Omega)$, we define
$$
\tilde {\mathbf T} \mathbf f: = \sum_{s=1}^n (-1)^{s-1}\sum_{1\le i_1<\cdots<i_s\le n} {\mathbf T}^{[i_1, \cdots, i_s]} {\mathbf f}. 
$$
Here for each $(i_1, \cdots, i_s)$ with $1\le i_1<\cdots<i_s\le n$, let $(z', w'')$ represents the  point whose $j$-th component is $z_j$ if $j\in \{i_1, \ldots, i_s\}$, and is $w_j$ otherwise. Then  for $w\in \Omega$,
\begin{align*}\label{ns}
     &{\mathbf T}^{[i_1, \cdots, i_s]}\mathbf f(w)\\
     :=& \sum_{k=1}^s\sum_{m=0}^{s-1}(-1)^m
\sum_{1\le j_1<\cdots<j_m\le i_s}\int \limits_{D_{j_1}\times \cdots\times D_{j_m}\times bD_{t_1}\times \cdots \times bD_{t_{s-m-1}}\times D_{i_k}} f_{i_k}(z', w'')\frac{\partial^m e_w^{k, i_1, \ldots, i_s}(z)}{\partial \bar{z}_{j_1}\cdots\partial \bar{z}_{j_m}},
\end{align*}
where the third sum is over all $1\le j_1<\cdots<j_m\le i_s $  such that
$$\{j_1, \ldots, j_m\}\cup\{t_1, \ldots, t_{s-m-1}\} = \{i_1, \ldots, \hat i_k, \ldots, i_s\},$$
 and
\begin{equation}\label{ee1}
    e_w^{k, i_1, \ldots, i_s}(z): = \left(\sum_{k=1}^s\prod_{l=1,l\neq k}^s |{w}_{i_l}-z_{i_l}|^2\right)^{-1}   {\prod_{l=1}^s S_{i_l}({w}_{i_l}, z_{i_l})\prod_{l=1, l\ne k}^{s} |{w}_{i_m}- z_{i_m}|^2}.  
\end{equation} 
For convenience, we suppress the corresponding measure elements  from integrals. 

In Proposition \ref{ee}, we have shown   that there exist some constants  $\alpha_r \in [0, 2)$ when $1\le r\le m+1$, and $\beta_r \in [0,1) $ when $   1\le r\le s-m-1$, such that 
\begin{equation}\label{G11} 
 \left |\frac{\partial^m e_w^{k, i_1,\ldots, i_s}}{\partial \bar{z}_{j_1}\cdots\partial \bar{z}_{j_m}}\right |\le   {C}{\prod_{r=1}^m|{z}_{j_r}-w_{j_r}|^{-\alpha_r}\prod_{r=1}^{s-m-1}|{z}_{t_r}-w_{t_r}|^{-\beta_r}|{z}_{{i_k}}-w_{i_k}|^{-\alpha_{m+1}}}
\in  L^1(R),
\end{equation}
where $R:={D_{j_1}\times \cdots\times D_{j_m}\times bD_{t_1}\times \cdots bD_{t_{s-m-1}}\times D_{i_k}}$. Thus $\tilde {\mathbf T} \mathbf f$ is well defined over $C_{(0,1)}(\bar\Omega)$.

\medskip

Our goal is to prove that $\tilde {\mathbf T} \mathbf f$ is the canonical solution to  $\bar\partial u = \mathbf f$ in the sense of distributions for all continuous $\bar\partial$-closed $(0,1)$ forms $\mathbf f$  up to $\bar\Omega$. Note that if the datum $\mathbf f\in C^{n-1}_{(0, 1)}(\bar\Omega)$, the proof of  Theorem \ref{C1} already implies that  $\tilde {\mathbf T} \mathbf f$ is equal to ${\mathbf T}\mathbf f$ and thus is the canonical solution. 
 The following proposition proves that $\tilde {\mathbf T} \mathbf f$  solves the $\bar\partial$-equation for continuous data in the sense of distributions, making use of a similar approximation argument as in \cite{PZ}.

\begin{pro}\label{main11}
There exists a positive constant $C$ depending only on $\Omega$ such that for any $\bar\partial$-closed $(0, 1)$ form $\mathbf f$ continuous up to $\bar\Omega$,  
  $\tilde {\mathbf T}\mathbf f $  is a continuous solution (in the sense of distributions) to  $\bar\partial u =\mathbf f$,  and satisfies
    $$
    \|\tilde {\mathbf T} \mathbf f\|_{ \infty }\le C \|\mathbf f\|_{ \infty }.  
    $$
\end{pro}

\begin{proof}
The fact that 
$\tilde {\mathbf T}\mathbf f \in C(\Omega)$ follows from (\ref{G11}) and the Dominated Convergence Theorem. Moreover, $\|\tilde {\mathbf T} \mathbf f\|_{\infty}\le C \|\mathbf f\|_{\infty}$ by (\ref{G11}). It remains to show that $\tilde {\mathbf T}\mathbf f$ solves $\bar\partial u =\mathbf f$ in $\Omega$ in the sense of distributions. 
  
  \medskip
  
 For each $j\in \{1, \ldots, n\}$, let $ \{D^{l}_j\}_{l=1}^\infty$   be a family of smooth   domains compactly contained in $D_j$ exhausting   $D_j$ in the sense of  Ne\v{c}as (see \cite{Ne}). Namely, \\
 a) for  $l  \in \mathbb N $, $bD^{l}_j$ is  $C^{\infty}$ and $ \dist(D^{l}_j, b D_j)<l^{-1}$;\\
 b) there exists a Lipschitz diffeomorphism  $h_j^{l}: \bar D_j\rightarrow \bar D_j^{l}$    and  functions $\omega_j^l: bD_j\rightarrow \mathbb R^+$ that are uniformly bounded,   $\omega_j^l\rightarrow 1$ a.e., such that  for any $f\in L^1( h^l_j(bD_j) ) $, the following change-of -variables formula holds.
 \begin{equation}\label{cv}
     \int_{bD_j}
f (h^l_j(z_j)) \omega^l_j(z_j) d\sigma_j(z_j) = \int_{ h^l_j(bD_j)}
f (z^l_j)\,d\sigma^l_j(z^l_j),
 \end{equation} 
where $d\sigma^l_j$ denotes arc-length measure on $bD^l_j$.  \\
 Denote by $\Omega^{l}:= D^{l}_1\times\cdots\times D^{l}_n$  the product of these planar domains, and $h^l(z):=(h_1^l(z_1), \ldots, h_n^l(z_n))$ a diffeomorphism from $z\in \bar\Omega$ to $\bar\Omega^l$.  Let ${e^{(l)}}$,  ${\mathbf T}^{(l)}$,  $\tilde {\mathbf T}^{(l)}$ and $( {\mathbf T}^{(l)})^{[i_1, \cdots, i_s]}$ stand for the corresponding  operators on $\Omega^{l}$ instead. Then $\tilde {\mathbf T}^{(l)} \mathbf f\in L^\infty(\Omega^{l})$. 
 
   \medskip
   
 For each  $l$, we adopt the mollifier argument to $\mathbf f\in C(\bar\Omega)$ and obtain $\mathbf f^\epsilon\in C^{\infty}( {\Omega^{l}}) \cap L^{\infty}(\Omega^{l})$  such that  $\|\mathbf f^\epsilon - \mathbf f\|_{L^\infty(\Omega^{l})}\rightarrow 0$  as $\epsilon\rightarrow 0$.  Since $\bar\partial \mathbf f^\epsilon =0$ on $\Omega^{l}$,   $\bar\partial \tilde {\mathbf T}^{(l)} \mathbf f^\epsilon =\bar\partial {\mathbf T}^{(l)} \mathbf f^\epsilon=\mathbf f^\epsilon$ in $ \Omega^{l}$ with  $ \tilde {\mathbf T}^{(l)} \mathbf f^\epsilon\in L^\infty(\Omega^{l})$ when $\epsilon$ is small. Furthermore, 
 \begin{equation*} \label{DD10}
     \|\tilde {\mathbf T}^{(l)} \mathbf f^\epsilon - \tilde {\mathbf T}^{(l)} \mathbf f\|_{L^\infty(\Omega^{l})} \le C_l \|\mathbf f^\epsilon - \mathbf f\|_{L^\infty(\Omega^{l})}\rightarrow 0,
 \end{equation*} as $\epsilon\rightarrow 0$. We thus  have that $\lim \limits_{\epsilon\rightarrow 0}\tilde {\mathbf T}^{(l)} \mathbf f^\epsilon$ exists almost everywhere on $\Omega^{l}$
and is equal to $\tilde {\mathbf T}^{(l)} \mathbf f \in L^\infty(\Omega^{l})$.

Given a test function $\phi\in C^\infty_0(\Omega)$ with a compact support $K$, let $l_0$ be such that $K \subset \Omega^{l_0}$. 
Denote by $\langle  \cdot, \cdot \rangle_{\Omega}$ and $\langle  \cdot, \cdot \rangle_{\Omega^{l_0}} $ the inner products in $L^2(\Omega)$ and $L^2(\Omega^{l_0})$, respectively. Let $\bar\partial^* $ be the formal adjoint of $\bar\partial$. 
For $l\ge l_0$,
\begin{equation}\label{22}
\langle \tilde {\mathbf T}^{(l)}\mathbf f, \bar\partial^*\phi\rangle_{\Omega^{l_0}} =\lim_{\epsilon \rightarrow 0} \langle \tilde {\mathbf T}^{(l)}\mathbf f^\epsilon, \bar\partial^*\phi \rangle_{\Omega^{l_0}}= \lim_{\epsilon \rightarrow 0}\langle \bar\partial {\mathbf T}^{(l)}\mathbf f^\epsilon, \phi \rangle_{\Omega^{l_0}} = \lim_{\epsilon \rightarrow 0} \langle\mathbf f^\epsilon, \phi \rangle_{\Omega^{l_0}} = \langle \mathbf f, \phi \rangle_{\Omega}.
\end{equation}

We further claim that \begin{equation}\label{11}
\langle \tilde {\mathbf T}\mathbf f, \bar\partial^*\phi \rangle_{\Omega}=\lim_{l\rightarrow \infty}  \langle  \tilde  {\mathbf T}^{(l)}\mathbf f, \bar\partial^*\phi  \rangle_{\Omega^{l_0}}.
\end{equation}
To get \eqref{11}, it suffices to show that 
\begin{equation}\label{dd7}
\langle   {\mathbf T}^{[i_1, \cdots, i_s]}\mathbf f, \bar\partial^*\phi \rangle_{\Omega^{l_0}}=\lim_{l\rightarrow \infty}\langle (   {\mathbf T}^{(l)})^{[i_1, \cdots, i_s]}\mathbf f, \bar\partial^*\phi \rangle_{\Omega^{l_0}}. \end{equation}
Without loss of generality, assume $(i_1, \ldots, i_s)=(1, \ldots, s)$, and  $(j_1, \ldots, j_m)=(1, \ldots, m) $ with $i_k= m+1$. The remaining cases can be treated similarly. Writing $z=(z', z'')\in \mathbb C^{s}\times \mathbb C^{n-s}$,  (\ref{dd7}) is equivalent to
\begin{equation}\label{do}
         \lim _{l\rightarrow \infty}\int \limits_{w\in K} \int \limits_{z'\in\Gamma^{l}}f_{{m+1}}(z', w'')\bar\partial^*\phi(w)D^m e^{(l)}(w', z')
         =  \int \limits_{w\in K}\int \limits_{z'\in\Gamma} f_{{m+1}}(z', w'')\bar\partial^*\phi(w)D^m e(w', z'),
\end{equation}
where $\Gamma=\Gamma_1\times \Gamma_2: = (D_{1}\times \cdots\times D_{m}\times D_{m+1}) \times (bD_{m+2}\times \cdots \times b D_{s})$, $D^m e(w', z'): = \frac{\partial^m e_w^{k, i_1,\ldots, i_s}(z)}{\partial \bar z_{1}\cdots\partial \bar z_{m}},$
and $\Gamma^{l}, D^m e^{(l)}$ are defined similarly with respect to each $l$.
Indeed, by the change-of-variables formula \eqref{cv}, we have
\begin{align*}
   &\int \limits _{w\in K} \int \limits_{z'\in \Gamma^{l}}f_{{m+1}}(z', w'')\bar\partial^*\phi(w)D^m e^{(l)}(w', z')\\
         =&  \int \limits_{(w, z')\in K \times \Gamma} f_{{m+1}}(h^{l}(z'), w''))\bar\partial^*\phi(w)D^m e^{(l)}(w', h^{l}(z'))\omega^l (z')=: \int \limits_{(w, z')\in  K \times \Gamma} F^{(l)}(w, z').
\end{align*}
 Here 
 $\omega^l (z'): = \prod_{j=1}^s\omega^l_j(z'_j)$ with  $\omega^l (z')\rightarrow 1$ as $l\rightarrow \infty$.

   \medskip
   
Notice that if  $w\in K\Subset \Omega$ and $z_k\in b D_k^{l}$,    then there exists some $\delta_0>0$ dependent only on $K$ such that when $l$ is large enough, one has  for $k = m+2,\ldots,  s$, 
  $$
  |z_k-w_k|\ge   \delta_0.
  $$
Hence  for  all  $(w, z')\in K \times \Gamma$,  there exist $ \alpha_j \in [0, 2)$, $\beta_0\ge 0$, and a positive constant $C$ independent of $l$ by  (\ref{G11}) such that 
 \begin{equation}\label{dd3}
  \left|  F^{(l)}(w, z') \right| \le  {C \delta_0^{-\beta_0}\prod_{j=1}^{m+1} |z_j-w_j|^{-\alpha_j}}\in     L^1( K \times \Gamma).  
 \end{equation}

On the other hand, we will show that 
 for  all $(w, z')\in K  \times \Gamma\setminus \cup_{j=1}^{m+1}\{z_j=w_j\} $, 
\begin{equation}\label{hhh}
    \lim_{l\rightarrow \infty} D^m e^{(l)}( w', h^{l}(z')) = D^m e( w', z').
    \end{equation}
Note that from the definition (\ref{ee1}) for $e$,  besides some explicit continuous functions, $D^m e^{(l)}$ off the diagonal involves only products of  the canonical solution kernels and the Bergman kernels along the first $s$ slices. 
In view of this, (\ref{hhh})  follows from the fact that the canonical solution kernel for the $\bar\partial$-equation and the Bergman kernel converge locally uniformly on planar domains (see Proposition \ref{DD} on the stability).
By the continuity of $\mathbf f$ and  the construction of $\Omega^{l}$,
$$
        \lim_{l\rightarrow \infty} F^{(l)}(w, z') =  f_{{m+1}}(z', w'')\bar\partial^*\phi(w)D^m e( w', z')
 $$
point-wisely on $ K \times \Gamma\setminus \cup_{j=1}^{m+1}\{z_j=w_j\} $. Therefore, (\ref{do})  follows from (\ref{dd3}), \eqref{hhh} and the Dominated Convergence Theorem.

 Finally, combining  (\ref{22}) with (\ref{11}),  we complete the proof by 
 deducing that
$$
\langle  \tilde {\mathbf T}\mathbf f, \bar\partial^*\phi  \rangle_{\Omega}=\lim_{l\rightarrow \infty}  \langle \tilde  {\mathbf T}^{(l)}\mathbf f, \bar\partial^*\phi  \rangle_{\Omega^{l_0}}= \langle \mathbf f, \phi \rangle_\Omega.
 $$

\end{proof}

 \begin{proof}[Proof of Theorem \ref{main}]  In view of Proposition \ref{main11}, we only need to show that for any given $\epsilon>0$, $| \langle \tilde {\mathbf T}\mathbf f, h \rangle |\le \epsilon$ for any $h\in A^2(\Omega)$ with $\|h\|_{L^2(\Omega)}=1$. Without loss of generality, we may further assume that  $\|\mathbf f\|_{L^2(\Omega)} =1$ and $\text{Vol}(\Omega)\le 1$.

From (\ref{dd3}) we know for all $w\in K$ that
 \begin{equation*}
  \int_{z'\in \Gamma^{l}}f_{{m+1}}(z', w'')D^m e^{(l)}(w', z')\le C
 \end{equation*}
 and
 $$
        \lim_{l\rightarrow \infty}  \int_{z'\in \Gamma^{l}}f_{{m+1}}(z', w'')D^m e^{(l)}(w', z')         = \int_{z'\in\Gamma} f_{{m+1}}(z', w'')D^m e(w', z').
$$
 So the Dominated Convergence Theorem guarantees that for any $ K \Subset \Omega$, as $l\rightarrow \infty$,
 \begin{equation}\label{DD8}
    \| \tilde {\mathbf T}\mathbf f -\tilde {\mathbf T}^{(l)}\mathbf f\|_{L^2(K)}\rightarrow 0.
 \end{equation}

For any given $\epsilon>0$, according to   the proof of Proposition \ref{main11} and (\ref{DD8}), there exist  two product domains $\Omega^{l_0}$ and $\Omega^{l_1}$  with component-wise $C^2$ boundaries and $\Omega^{l_0}\Subset \Omega^{l_1}\Subset\Omega$, and a $\bar\partial$-closed form $\mathbf g \in C^{n-1}(\overline{\Omega^{l_1}})$ such that the followings  hold.
  
  \medskip
  
a)  \begin{equation}\label{dd2}
     \|h\|_{L^2(\Omega\setminus \Omega^{l_0})}\le   {(12C_0)^{-1}}\epsilon.
\end{equation} 
    Here  $C_0$ is the $L^\infty$ bound for both $\mathbf{\tilde T}$ on $\Omega$ and   $\mathbf{\tilde T}^{(l_1)}$ on $\Omega^{l_1}$. 

\medskip

b) 
$
 \|\tilde {\mathbf T}\mathbf f - \tilde {\mathbf T}^{(l_1 )}\mathbf f\|_{L^2(\Omega^{l_0})} \le 6^{-1}{\epsilon}.
$

\medskip

c) $\|\mathbf f-\mathbf g\|_{L^\infty(\Omega^{l_1})} \le (6C_0)^{-1}{\epsilon}$.

\medskip
      
 As  consequences of the construction,  we obtain
\begin{equation}\label{DD11}
\begin{split}
&\langle  \tilde {\bf T}^{(l_1)}\mathbf g, h \rangle_{\Omega^{l_1}} =0;\\
&\| \tilde {\mathbf T}^{(l_1)}\mathbf f \|_{L^2(\Omega^{l_1})}\le  \| \tilde {\mathbf T}^{(l_1)}\mathbf f \|_{L^\infty(\Omega^{l_1})}\le C_0\|f\|_{L^\infty(\Omega)} = C_0, \quad \text {and similarly } \quad \| \tilde {\mathbf T}\mathbf f \|_{L^2(\Omega)}\le C_0.\\
  &\|\tilde {\bf T}^{(l_1)}\mathbf f - \tilde {\bf T}^{(l_1)}\mathbf g\|_{L^2(\Omega^{l_1})}\le  C_0 \| \mathbf f -  \mathbf g\|_{L^\infty(\Omega^{l_1})} \le  6^{-1}{\epsilon} .
\end{split}
\end{equation}
 Here the first identity is due to the facts that $h|_{\Omega^{l_1}}\in A^2(\Omega^{l_1})$ and $\tilde {\bf T}^{(l_1)}$ is the canonical solution operator on the space $C^{n-1}(\overline{\Omega^{l_1}})$. 
Combining (\ref{DD11}) with b)-c) and making use of H\"older inequality, one has 
\begin{align*} 
 | \langle \tilde {\bf T}\mathbf f, h \rangle_{\Omega^{l_1}}|  
 \le &| \langle \tilde {\bf T}\mathbf f - \tilde {\bf T}^{(l_1)}\mathbf f, h \rangle _{\Omega^{l_1}}| +| \langle \tilde {\bf T}^{(l_1)}\mathbf f - \tilde {\bf T}^{(l_1)}\mathbf g, h \rangle_{\Omega^{l_1}}| \\
  \le &\big  \|\tilde {\bf T}\mathbf f-\tilde {\mathbf T}^{(l_1)}\mathbf f  \big  \|_ {L^2 \left (\Omega^{l_0} \right )}+ \big   \|\tilde {\bf T}\mathbf f-\tilde {\mathbf T}^{(l_1)}\mathbf f  \big  \|_{L^2\left (\Omega^{l_1}\setminus \Omega^{l_0} \right)}  \|h\|_{L^2 \left (\Omega^{l_1}\setminus \Omega^{l_0} \right )}+  \big  \|\tilde {\bf T}^{(l_1)}\mathbf f - \tilde {\bf T}^{(l_1)}\mathbf g  \big \|_{L^2 \left (\Omega^{l_1} \right )} \\
   \le & 6^{-1}{\epsilon}+ {(12C_0)^{-1} \epsilon}\left(  \|\tilde{\bf T}  \mathbf f\|_{ L^2(\Omega)} + \|\tilde {\mathbf T}^{(l_1)}\mathbf f\|_{L^2 \left(\Omega^{l_1}\right)}  \right) + 6^{-1}{\epsilon}
   \, =   \, 2^{-1}{\epsilon}. 
\end{align*}
Together with (\ref{dd2}), we finally obtain
\begin{equation*}
    | \langle \tilde {\bf T}\mathbf f, h \rangle_\Omega|\le | \langle \tilde {\bf T}\mathbf f, h \rangle_{\Omega^{(l_1)}}|+  | \langle \tilde {\bf T}\mathbf f, h \rangle_{\Omega\setminus \Omega^{(l_1)}}|\le 2^{-1}{\epsilon}+(12C_0)^{-1}  \epsilon\|\tilde {\bf T}\mathbf f\|_{L^2(\Omega)}\le \epsilon
\end{equation*}
and thus have proved the main theorem completely. 

\end{proof}

\begin{remark}\label{re}
After the first version of this paper was circulated, we noticed that since  $S_j(w_j, \cdot)=0$ on $bD_j$ by \eqref{S0},  the  operator $\tilde {\mathbf T}$ can be further reduced to 
\begin{align*}\label{ns1}
     {\mathbf T^\sharp} \mathbf f: =& \sum_{s=1}^n \sum_{1\le i_1<\cdots<i_s\le n}\sum_{k=1}^s
\int \limits_{D_{i_1}\times \cdots\times D_{i_s}} f_{i_k}(z', w'')\frac{\partial^{s-1} e_w^{k, i_1, \ldots, i_s}(z)}{\partial \bar{z}_{i_1}\cdots\partial \bar z_{i_{k-1}}\partial \bar z_{i_{k+1}}\cdots\partial \bar{z}_{i_s}}.
\end{align*}
Write $R:={D_{i_1}\times \cdots\times D_{i_s}}$ and $\Omega = R\times E$ with an abuse on  the order of slices. For a.e. $w''\in E$, making use of our crucial estimate \eqref{G11} and   the   Young's inequality on $R$, one can get for all $1\le p \le \infty$,
$$  \left\| \int \limits_{R} f_{i_k}(z', w'')\frac{\partial^{s-1} e_w^{k, i_1, \ldots, i_s}(z)}{\partial \bar{z}_{i_1}\cdots\partial \bar z_{i_{k-1}}\partial \bar z_{i_{k+1}}\cdots\partial \bar{z}_{i_s}}\right\|_{L^p(R)}\le C\| \textbf f(\cdot, w'')\|_{L^p(R)}, $$ 
where $C$ is independent of $p$. Thus
\begin{equation*}
    \begin{split}
       \left\| \int \limits_{R} f_{i_k}(z', w'')\frac{\partial^{s-1} e_w^{k, i_1, \ldots, i_s}(z)}{\partial \bar{z}_{i_1}\cdots\partial \bar z_{i_{k-1}}\partial \bar z_{i_{k+1}}\cdots\partial \bar{z}_{i_s}}\right\|_{L^p(\Omega)}  =& \left \| \left\| \int \limits_{R} f_{i_k}(z', w'')\frac{\partial^{s-1} e_w^{k, i_1, \ldots, i_s}(z)}{\partial \bar{z}_{i_1}\cdots\partial \bar z_{i_{k-1}}\partial \bar z_{i_{k+1}}\cdots\partial \bar{z}_{i_s}}\right\|_{L^p(R)}\right\|_{L^p(E)}\\
       \le &C\left\| \|\textbf f(\cdot, w'')\|_{L^p(R)}  \right\|_{L^p(E)}  = C\|\textbf f\|_{L^p( \Omega)}.
    \end{split}
\end{equation*}
Hence $\mathbf T^\sharp $ is well-defined on $L^p(\Omega)$, with $ \|\mathbf T^\sharp \mathbf f\|_{ L^p(\Omega) }\le C \|\mathbf f\|_{ L^p(\Omega) }$ for some $C$  independent of $p$. Proposition \ref{main11} can be readily simplified to prove that $\mathbf T^\sharp $  is a solution operator for   $L^p(\Omega)$ data accordingly.

\medskip
This simplified formulation was also noted in Yuan \cite{Yuan}, who extended these methods to study the $L^p$ estimates. See also Li \cite{Li} for a different approach to Kerzman's problem.

\end{remark}

 \medskip

Lastly, an immediate consequence of Theorem \ref{main} is the following estimate concerning the Bergman projection, which generalizes Corollary \ref{cor=dim1} to product domains.

\begin{cor} \label{cor} Under the same assumptions as in Theorem \ref{main}, 
let $\mathbf P$ be the Bergman projection on $\Omega$. Then there exists a positive constant $C$ depending only on $\Omega$ such that for any $
u\in C  (\bar\Omega)  \cap W^{1, \infty}(\Omega)$,
$$
\|\mathbf P u\|_{ \infty }\le \|u \|_{ \infty } + C \|\bar\partial u\|_{ \infty }.
$$
\end{cor}

\bigskip
\noindent{\LARGE{\bf Appendices}}
 
\begin{appendices} 
  
\section{Proof of formula (\ref{L=H+I})} \label{AppA}

 We shall follow Barletta-Landucci and  verify (\ref{L=H+I}), correcting a minor error in \cite{BL91} (where an additional negative sign was mistakenly introduced while transforming an integral from being  in terms of $d\zeta$ to $ds_\zeta$). 
 
 Fix $z\ne w\in D$, and let $\epsilon$ be defined in (\ref{epsilon}), $\sigma < \min\{\delta(z), \epsilon\}$. Then $G(z, \cdot)$ and $1/(\cdot-w)$ are both harmonic on $D_{\epsilon, \sigma}: = D\setminus (B(z, \sigma)\cup B(w, \epsilon))$.   From the definition of $L$ in (\ref{L def}) and the fact that $G=0$ on $bD$, a direct computation    yields
 \begin{equation*}
\begin{split}
-2\pi i L(w, z)=& \int_{ b D} \frac{1}{\zeta-w} \frac{\partial G(z, \zeta)}{\partial \vec{n}_\zeta }ds_\zeta\\  
 =&\int_{ b D} \frac{1}{\zeta-w} \frac{\partial G(z, \zeta)}{\partial \vec{n}_\zeta } - G(z, \zeta) \frac{\partial }{\partial \vec{n}_\zeta }\left(\frac{1}{\zeta-w}\right)ds_\zeta\\
  =& \int_{bB(z, \sigma)}  \frac{1}{\zeta-w} \frac{\partial G(z, \zeta)}{\partial \vec{n}_\zeta } - G(z, \zeta) \frac{\partial }{\partial \vec{n}_\zeta }\left(\frac{1}{\zeta-w}\right)ds_\zeta+\\
  &+\int_{bB(w, \epsilon)}  \frac{1}{\zeta-w} \frac{\partial G(z, \zeta)}{\partial \vec{n}_\zeta } - G(z, \zeta) \frac{\partial }{\partial \vec{n}_\zeta }\left(\frac{1}{\zeta-w}\right)ds_\zeta =: I_1+I_2.
\end{split}
\end{equation*}
Here the third equality is due to Green's second identity, in view of the harmonicity of  $\frac{1}{\cdot-w} $ and $G(z, \cdot)$  in $D_{\epsilon, \sigma}$.

Since $\rho(z, \cdot): = G(z, \cdot) + \frac{1}{2\pi}\log|z -\cdot|$ and $\frac{1}{\cdot-w}$ are both harmonic in $B(z, \sigma)$, we  similarly have 
 $$ \int_{bB(z, \sigma)}  \frac{1}{\zeta-w} \frac{\partial \rho(z, \zeta)}{\partial \vec{n}_\zeta } - \rho (z, \zeta) \frac{\partial }{\partial \vec{n}_\zeta }\left(\frac{1}{\zeta-w}\right)ds_\zeta =0.$$
 Hence
 \begin{equation*}
 \begin{split}
          I_1 = &-\frac{1}{2\pi}\left(\int_{bB(z, \sigma)}  \frac{1}{\zeta-w} \frac{\partial \log|z-\zeta|}{\partial \vec{n}_\zeta }ds_\zeta - \int_{bB(z, \sigma)}\log|z- \zeta| \frac{\partial }{\partial \vec{n}_\zeta }\left(\frac{1}{\zeta-w}\right)ds_\zeta\right)\\
          =&-\frac{1}{2\pi}\left(\int_0^{2\pi} \frac{\sigma}{z+\sigma e^{it}-w} \frac{\partial \log\sigma}{\partial \sigma }dt - \int_{bB(z, \sigma)}\sigma\log\sigma \frac{\partial }{\partial \sigma }\left(\frac{1}{z+\sigma e^{it}-w}\right)dt\right)\\
          =& -\frac{1}{2\pi}\left(\int_0^{2\pi} \frac{1}{z+\sigma e^{it}-w} dt - \sigma\log\sigma \int_{bB(z, \sigma)} \frac{\partial }{\partial \sigma }\left(\frac{1}{z+\sigma e^{it}-w}\right)dt\right).
 \end{split}
 \end{equation*}
 Letting $\sigma$ go to $0$, we have
 $$I_1 \rightarrow -\frac{1}{z-w}.  $$

 Finally we estimate  $I_2$ on $bB(w, \epsilon)$. Note that
 $$ \int_{bB(w, \epsilon)}  \frac{1}{\zeta-w} \frac{\partial G(z, \zeta)}{\partial \vec{n}_\zeta } ds_\zeta = \frac{1}{\epsilon^2}\int_{bB(w, \epsilon)}  (\bar\zeta-\bar w) \frac{\partial G(z, \zeta)}{\partial \vec{n}_\zeta } ds_\zeta.$$
 Because $\bar\cdot-\bar w  $ and $G(z, \cdot)$ are both harmonic in $B(w, \epsilon)$, we use Green's second identity again to get 
 $$ \int_{bB(w, \epsilon)}  \frac{1}{\zeta-w} \frac{\partial G(z, \zeta)}{\partial \vec{n}_\zeta } ds_\zeta = \frac{1}{\epsilon^2}\int_{bB(w, \epsilon)}  G(z, \zeta)  \frac{\partial (\bar\zeta-\bar w)}{\partial \vec{n}_\zeta } ds_\zeta$$
Therefore,
 \begin{equation*}
     \begin{split}
         I_2 = & \frac{1}{\epsilon}\int_0^{2\pi}G(z, w+\epsilon e^{it})   \frac{\partial (\epsilon e^{-it})}{\partial \epsilon }dt -\epsilon \int_0^{2\pi} G(z, w+\epsilon e^{it}) \frac{\partial }{\partial \epsilon}\left(\frac{1}{\epsilon e^{it}}\right)dt\\
         =&\frac{2}{\epsilon}\int_0^{2\pi} G(z, w+\epsilon e^{it}) e^{-it}dt. 
     \end{split}
 \end{equation*}
Altogether, we have
$$L(w, z) = \frac{1}{2\pi i(z-w)} -\frac{1}{\epsilon\pi i}\int_0^{2\pi} G(z, w+\epsilon e^{it}) e^{-it}dt. $$

\section{Estimates for Green's function} \label{AppB}

We 
provide proofs to Lemma \ref{Kerzman-Green}  based on \cite{Ker76} and Lemma \ref{Kerzman-Green2} in this section. Kerzman's result  on Green's function serves as a foundational estimate, from  which upper bounds for  the canonical solution kernel of $\bar\partial$ and its derivatives are obtained in Section 2 and 3, respectively.

On a bounded domain $D \subset \mathbb{C}$ with $C^{2}$ boundary, Green's function $G$ satisfies the following properties:

\vspace{0.15cm}

i) $ G(z,w)=G(w, z) $ are positive $C^2$ functions for $(z, w) \in D \times D \setminus \{z=w\}$. 

\vspace{0.15cm}

ii) For each fixed $w\in D$,  $ G(z, w)$ is harmonic in $z \in D \setminus \{w\}$, and $G(z, w) = 0$ when $ z \in bD$.
\vspace{0.15cm}

iii) For each fixed $w\in D$,  $ G(z,w)+(2\pi)^{-1}\log|z-w|$ is harmonic in $z\in D$.
 
\vspace{0.15cm}

Fix $w\in D$ and consider $\rho(z, w):= G(z, w) + (2\pi)^{-1}  \log |z-w|$.  
When $z \in bD$, $\rho(z, w)= (2\pi)^{-1}  \log |z-w| \le  (2\pi)^{-1}  \log d$, where $d$ is the diameter of $D$. By the Maximum Principle, we know that for $(z, w) \in D\times D$ off the diagonal,
$$
G(z, w) +\frac{1}{2\pi}\log |z-w|\le \frac{1}{2\pi}\log d,
$$
which yields that
\begin{equation}\label{G<}
0<G(z, w) \le  \frac{1}{2\pi}\log\frac{d}{|z-w|}.
\end{equation}
\medskip

\begin{proof}[Proof of Lemma \ref{Kerzman-Green}]

Since $G$ is symmetric in $z$ and $w$ and $G(z, \cdot)=0$ on $bD$, we shall prove the first part of (\ref{G}) only, for any fixed pole $z\in D$. By rotation and translation if necessary, $w$ is on the negative $y$-axis and $0\in bD$ such that $\delta(w) = |w|$. Since $D$ has $C^2$ boundary, it satisfies the exterior ball condition with a uniform radius $r\leq d$.
If $\delta(w) \ge (8d)^{-1}r|z-w|$, then (\ref{G}) follows from \eqref{G<}. So we assume that $\delta(w) \le (8d)^{-1}r|z-w|$.   
Let $\alpha :=   \frac{r|z-w|}{4d}< 4^{-1}r $ and define the region $R:= B(0, \alpha)\setminus B(\alpha i, \alpha)$. On $R$, consider  the harmonic function $$E(\zeta): =\Real    \{ (\zeta -\alpha i)^{-1}\zeta \}.$$ Then 
$E\ge {c_0}$ on $bB(0, \alpha)\setminus \overline{B(\alpha i, \alpha)}$ for some universal constant
 $$
 {c_0} :=\inf_{\zeta\in bB(0, 1)\setminus \overline{B( i, 1)}}\Real    \frac{\zeta}{\zeta -i}>0,
 $$
 which is particularly independent of $\alpha$. Moreover, at $\zeta=w$,
\begin{equation}\label{E2}
    E(w)\le \frac{|w|}{\alpha}= \frac{ 4d\delta(w)}{  r|z-w|}.
\end{equation}
Consider the sub-region $\tilde R: = R\cap D$, where the function  $E(\zeta)\log\frac{2d}{|z-w|}$ is harmonic in $\zeta$.

\vspace{0.15cm}

We assert that
\begin{equation}\label{claim}
E(\zeta) \log \frac{2d}{|z-w|}  \ge  2\pi {c_0}G(z, \zeta),
\end{equation}
for $\zeta\in b\tilde R\equiv   (R\cap bD) \cup (bR\cap D)$.

\vspace{0.15cm}

If $\zeta\in R\cap bD$, then $E(\zeta)\log\frac{2d}{|z-w|}\ge 0=2\pi {c_0}G(z, \zeta)$.

\vspace{0.15cm}

If $\zeta\in bR\cap D$,
$$
|z-\zeta|\ge |z-w|-|w-\zeta|\ge |z-w|-2\alpha\ge \frac{|z-w|}{2},
$$
and thus by \eqref{G<} again,
$$
G(z, \zeta) \leq  \frac{1}{2\pi}\log \frac{2d}{|z-w|}.
$$
On the other hand, since $bB(\alpha i, \alpha)\cap  D =\emptyset$ by the choice of $\alpha$, we have $\zeta \in bB(0, \alpha)\cap D\subset bB(0, \alpha)\setminus \overline{B(\alpha i, \alpha)}$ where $E\ge {c_0}$. Therefore, for $\zeta \in  bR\cap D$, we get \eqref{claim} and thus have verified the above assertion.

By the Maximum Principle, \eqref{claim} holds true for all $\zeta \in \tilde R$. In particular at $\zeta=w$, one has
\begin{equation}\label{E}
    E(w)\log\frac{2d}{|z- w|}\ge 2\pi {c_0}G(z, w).  \end{equation}
Now, the first part of \eqref{G} follows from (\ref{E}) and (\ref{E2}), by choosing $C$ to be a constant dependent only on $d, c_0$ and $r$. 

\vspace{0.15cm}
 
 For \eqref{g2}, it follows directly from \eqref{G} when $\delta(z)\ge (8d)^{-1}r|z-w|$. Otherwise, we argue in the same way as we prove \eqref{G}, and use \eqref{G} in the place of \eqref{G<} to show that the assertion
$$  
C  \frac{E(\zeta)\delta(z) }{|z-w|}\log\frac{2d}{|z-w|}\ge {c_0} G(z, \zeta)
$$
holds true for $\zeta\in b\tilde R$. By the Maximum Principle we further know that the above inequality holds true on all $\tilde R$, and particularly at $\zeta=w$. Using (\ref{E2}) again, we get \eqref{g2}.

\end{proof}

\begin{proof}[Proof of Lemma \ref{Kerzman-Green2}]
The symmetry of $G$ implies that the gradient in the first part of \eqref{Gd} could be taken with respect to either $z$ or $w$. Fix  $w\in D$, and for  $z_0 \in D \setminus \{w\}$ let $\epsilon =2^{-1} \min\{ {\delta(z_0)},  {|z_0-w|} \}>0$. Then $$|z-w|\ge |z_0-w|-|z_0-z|\ge \frac{|z_0-w|}{2}, \quad \text{and} \quad \delta(z)\le \delta(z_0)+|z_0-z| \le 2\delta(z_0)$$
   for all $z\in B(z_0, \epsilon)$. Since $G(\cdot, w)$ is harmonic on $B(z_0, \epsilon)$, by the Mean Value theorem of harmonic functions (cf. \cite[(2.31) on p. 22]{GT}),  it follows that
  \begin{equation*}
     |\nabla_z G(w, z_0)|\le \frac{2}{\epsilon}\sup_{z\in B(z_0, \epsilon)}|G(z, w)|.
 \end{equation*}
 When $\epsilon = 2^{-1} {\delta(z_0)} $, we use \eqref {G} to get
$$
|\nabla_z G(w, z_0)|\le \frac{4}{\delta(z_0)} \frac{4C \delta(z_0)}{|z_0-w|}\log \frac{2d}{|z_0-w|};
$$
when $\epsilon = 2^{-1} {|z_0-w|} $, we use \eqref{G<} to get
 $$
 |\nabla_z G(w, z_0)|\le \frac{4}{|z_0-w|} \frac{1}{2\pi}\log \frac{2d}{|z_0-w|}.
 $$
Then, the first part of \eqref {Gd} for $z\in D\setminus\{w\}$ is proved; 
  for $z\in bD$ it follows from the second part of \eqref{G} and the fact that $G( w, \cdot)$  vanishes at $bD$.
Similarly, to obtain the second part of \eqref {Gd}, we  use \eqref{g2} and \eqref{G} in the places of \eqref {G} and \eqref{G<}, respectively.

\end{proof}

\end{appendices}

\subsection*{Funding}

{\fontsize{11.5}{10}\selectfont

The research of the first author was supported by AMS-Simons travel grant. The research of the third author was partially supported by NSF DMS-1501024.
}

\bibliographystyle{alphaspecial}

\begin{thebibliography}{HD} 


{\fontsize{11}{11}\selectfont



   
\bibitem{BL91} E. Barletta and M. Landucci, \emph{Optimal $L^{\infty}$ estimates for the canonical solution of the CR-equation in the non-smooth case}, Complex Variables Theory Appl. {\bf 16} (1991), 93--106.

\bibitem {Bell} {S. R. Bell}, \emph{The Cauchy transform, potential theory and conformal mapping}. 2nd edition. Chapman \& Hall/CRC, Boca Raton, FL, 2016.

\bibitem{Be93} {B. Berndtsson}, \emph{A smooth pseudoconvex domain in $\mathbb C^2$ for which $L^{\infty}$-estimates for $\bar \partial$ do not hold}, Ark. Mat. {\bf 31} (1993), 209--218.

\bibitem {Be97} {B. Berndtsson}, {\it Uniform estimates with weights for the $\bar \partial- $equation}, J. Geom. Anal. {\bf 7} (1997), 195--215.

\bibitem {Be01} B. Berndtsson, \emph{Weighted estimates for the $\bar\partial$-equation}, Complex Analysis and Complex Geometry, de Gruyter, 43--57, 2001.


\bibitem{B} J.  Bertrams, \emph{Randregularit\"at von L\"osungen der $\bar\partial$-Gleichung auf dem Polyzylinder und zweidimensionalen
analytischen Polyedern} (German), Bonner Math. Schriften {\bf176} (1986), 1--164.



\bibitem{Boas} H. P. Boas,  \emph{The Lu Qi-Keng conjecture fails generically}, 
Proc. Amer. Math. Soc. {\bf124} (1996), 2021--2027.

\bibitem{CS11} D. Chakrabarti and M.-C. Shaw,  \emph{The Cauchy-Riemann equations on product domains}, Math. Ann. {\bf 349} (2011), 977--998.
 
 \bibitem{CZ} B. Chen and J. Zhang, {\em A remark on the Bergman stability}, Proc. Amer. Math. Soc. {\bf 128} (2000), 2903--2905.
 


\bibitem{DLT} R. X. Dong, S.-Y. Li and J. N. Treuer, \emph{Sharp pointwise and uniform estimates for $\bar\partial$}, Anal. PDE {\bf 16} (2023), 407--431.

\bibitem{FP19} M. Fassina and Y. Pan, \emph{Supnorm estimates for $\bar\partial$ on product domains in $\mathbb C^n$}, Acta Math. Sin. {\bf 40} (2024), 2307--2323. 


\bibitem{F-L-Zh11} J. E. Forn\ae ss, L. Lee and Y. Zhang, \emph{On suporm estimates for $\bar\partial$ on infinite type convex domains in $\mathbb C^2$}, J. Geom. Anal. {\bf 21} (2011), 495--512.

\bibitem{FS93} J. E. Forn\ae ss and N. Sibony, \emph{Smooth pseudoconvex domains in $\mathbb C^2$ for which the corona theorem and $L^p$ estimates for $\bar\partial$ fail}, Complex analysis and geometry, 209--222, Univ. Ser. Math., Plenum, New York, 1993.

\bibitem{GT} {D. Gilbarg and N. S. Trudinger}, \emph{Elliptic partial differential equations of second order} Reprint of the 1998 edition. Classics Math., Springer-Verlag, Berlin, 2001.
 
\bibitem{G-L70} {H. Grauert and I. Lieb}, \emph{Das Ramirezsche Integral Und Die L\"osung Der Gleichung $\bar{\partial}f=\alpha$ Im Bereich Der Beschr\"ankten Formen} (German), Rice Univ. Stud. {\bf 56} (1970), 29--50.

\bibitem{GSS19} D. Grundmeier, L. Simon and B. Stens\o nes, \emph{Sup-norm estimates for $\overline{\partial}$}, Pure Appl. Math. Q. {\bf 18} (2022), 531--571.


\bibitem{Hen70} {G. M. Henkin}, \emph{Integral representation of functions in strictly pseudoconvex domains and applications to the $\bar{\partial}-$problem} (Russian), Mat. Sb. (N.S.)  {\bf 82(124)} (1970), 300--308. English translation in Math. USSR Sb. {\bf 11} (1970), 273--281.

\bibitem{Hen71} {G. M. Henkin}, \emph{Uniform estimates for solutions to the $\bar{\partial}$-problem in Weil domains} (Russian), Uspehi Mat. Nauk {\bf 26} (1971), 211--212.


\bibitem{Kenig} {C. Kenig}, \emph{Harmonic analysis techniques for second order elliptic boundary value problems}, CBMS Regional Conf. Ser. in Math., 83. AMS, Providence, RI, 1994. 



\bibitem{Ker71} {N. Kerzman}, \emph{H\"{o}lder and $L^p$ estimates for solutions of $\bar{\partial}u=f$ in strongly pseudoconvex domains}, Comm. Pure. Appl. Math. {\bf 24} (1971) 301--379.


\bibitem{Ker73} {N. Kerzman}, \emph{Remarks on estimates for the $\bar{\partial}$-equation}, L'Analyse Harmonique dans le Domaine Complexe, 111--124, Lecture Notes in Math., 336, Springer, Berlin, Heidelberg, 1973.

\bibitem{Ker76} {N. Kerzman}, \emph{Topics in complex analysis: two estimates for Green's function}, unpublished notes, MIT, 1975/1976.

 
\bibitem {K76} S. G. Krantz, \emph{Optimal Lipschitz and $L^p$ estimates for the equation $\bar{\partial}u = f$ on strongly pseudoconvex domains}, Math. Ann. {\bf 219} (1976), 233--260.

\bibitem {K01} S. G. Krantz, \emph{Function Theory of Several Complex Variables. Second Edition}, AMS Chelsea Publishing, Providence, RI, 2001.

\bibitem {L75} M. Landucci, \emph{On the projection of $L^2(D)$ into $H(D)$}, Duke Math. J. {\bf 42} (1975), 231--237.

\bibitem{Li} S.-Y. Li, \emph{Solving the Kerzman's problem on the sup-norm estimate for $\bar\partial$ on product domains}, Trans. Amer. Math. Soc.  {\bf 377} (2024), 6725--6750.


\bibitem {LTL93} C. Laurent-Thi\`{e}baut and J. Leiterer, \emph{Uniform estimates for the Cauchy-Riemann equation on q-convex wedges}, Ann. Inst. Fourier (Grenoble) {\bf 43} (1993), 383--436.

\bibitem{LR} {I. Lieb and R. M. Range}, \emph{Integral representations and estimates in the theory of the $\bar \partial$-Neumann problem}, Ann.  Math. {\bf 123} (1986), 265--301.



\bibitem{Ne}{J.  Ne\v{c}as}, {\em Direct Methods in the Theory of Elliptic Equations,} Springer-Verlag, Berlin Heidelberg, 2012, xvi+372 pp.

\bibitem{PZ} Y. Pan and Y. Zhang, \emph{H\"older  estimates for the $\bar\partial$ problem for  $(p,q)$ forms on product domains}, Internat. J. Math.  {\bf 32} (2021),  20 pp.

\bibitem{R78} R. M. Range, \emph{On H\"{o}lder estimates for $\bar\partial u = f$ on weakly pseudoconvex domains}, Several complex variables (Proc. Int. Conf. Cortona, 1976/1977), 247--267, Scuola Norm. Sup. Pisa, 1978.

\bibitem{R86} R. M. Range, \emph{Holomorphic functions and integral representations in several complex variables}, Grad. Texts in Math., 108, Springer-Verlag, New York, 1986. 2nd corrected printing 1998.

\bibitem{R90} R. M. Range, \emph{Integral kernels and H\"{o}lder estimates for $\bar\partial$ on pseudoconvex domains of finite type in $\mathbb C^2$}, Math. Ann. {\bf 288} (1990), 63--74.

\bibitem{R20} M. Range, \emph{Integral Representations in Complex Analysis: From Classical Results to Recent Developments}, in: Breaz D., Rassias M. (eds) Advancements in Complex Analysis, 449--471, Springer, Cham, 2020.

\bibitem{R-S72} {R. M. Range and Y.-T. Siu}, \emph{Uniform estimates for the $\bar{\partial}-$equation on intersections of strictly pseudoconvex domains}, Bull. Amer. Math. Soc. {\bf 78} (1972), 721--723.

\bibitem{R-S73} R. M. Range and Y.-T. Siu, \emph{Uniform estimates for the $\bar{\partial}-$equation on domains with piecewise smooth strictly pseudoconvex boundaries}, Math. Ann. {\bf 206} (1973), 325--354.

\bibitem{HR} A. V. Romanov and G. M. Henkin, \emph{Exact H\"{o}lder Estimates for the Solutions of the $\bar\partial$-equation} (Russian), Math. USSR Izv. {\bf 5} (1971), 1180--1192.

\bibitem{SV14} {A. P. Schuster and D. Varolin}, \emph{New estimates for the minimal $L^2$ solution of  $\bar{\partial}$ and applications to geometric function theory in weighted Bergman spaces}, J. reine angew. Math. {\bf 691} (2014), 173--201.

\bibitem{Sib80} {N. Sibony}, \emph{Un exemple de domaine pseudoconvexe regulier o\`u l'\'equation $\bar{\partial}u=f$ n'admet pas de solution born\'{e}e pour f born\'{e}e} (French), Invent. Math. {\bf 62} (1980), 235--242.

\bibitem{Yuan}{Y. Yuan}, \emph{ Uniform estimates of the Cauchy-Riemann equation on product domains}. arXiv:2207.02592.


}

\end{thebibliography}

\fontsize{11}{9}\selectfont

\vspace{0.5cm}

\noindent xindong.math@outlook.com, 

\vspace{0.2 cm}

\noindent Department of Mathematics, University of Connecticut, Stamford, CT 06901-2315, USA

\vspace{0.4cm}

\noindent pan@pfw.edu,

\vspace{0.2 cm}

\noindent Department of Mathematical Sciences, Purdue University Fort Wayne, Fort Wayne, IN 46805-1499, USA

\vspace{0.4cm}

\noindent zhangyu@pfw.edu,

\vspace{0.2 cm}

\noindent Department of Mathematical Sciences, Purdue University Fort Wayne, Fort Wayne, IN 46805-1499, USA
 
\end{document}